\documentclass[10pt]{amsart}

\usepackage{array}
\usepackage{cancel}
\usepackage[british]{babel}
\usepackage{pifont,subfigure,graphicx}
\usepackage{caption}
\usepackage[colorlinks,citecolor=blue,urlcolor=blue]{hyperref}
\usepackage{url}
\usepackage{amsthm,amsmath,amssymb,amsfonts,mathrsfs,amsfonts,amstext,amsopn}
\usepackage{mathtools}
\usepackage{tikz}
\usepackage[normalem]{ulem}
\usepackage{hhline}
\usepackage[version=3]{mhchem}
\usepackage{texdraw}
\usepackage{extarrows}
\usepackage{bbm}
\usepackage{enumitem}
\usepackage[all]{xy}

\usepackage[tableposition=top]{caption}
\usepackage{longtable}
\usepackage{pgfplots}
\pgfplotsset{compat = newest}
\usetikzlibrary{decorations.fractals}
\usetikzlibrary{positioning}
\usetikzlibrary{arrows}

\numberwithin{equation}{section}
\theoremstyle{plain}
\newtheorem{theorem}{Theorem}[section]
\newtheorem{lemma}[theorem]{Lemma}
\newtheorem{corollary}[theorem]{Corollary}
\newtheorem{proposition}[theorem]{Proposition}

\theoremstyle{definition}

\newtheorem{example}[theorem]{Example}
\theoremstyle{remark}

\newcommand{\tS}{\text{S}}
\newcommand{\ka}{\kappa}
\newcommand{\lf}{\lfloor}
\newcommand{\rf}{\rfloor}
\newcommand{\lc}{\lceil}
\newcommand{\rc}{\rceil}
\newcommand{\lt}{\left}
\newcommand{\rt}{\right}

\newcommand{\om}{\omega}
\newcommand{\N}{\mathbb{N}}

\newcommand{\R}{\mathbb{R}}
\newcommand{\Z}{\mathbb{Z}}
\newcommand{\cF}{\mathcal{F}}

\newcommand{\cT}{\Omega}

\newcommand{\si}{{\sf i}}
\newcommand{\so}{{\sf o}}

\newcommand{\sL}{L}
\newcommand{\sU}{U}

\newcommand{\sN}{{\sf N}}
\newcommand{\sE}{{\sf E}}
\newcommand{\sP}{{\sf P}}
\newcommand{\sQ}{{\sf Q}}
\newcommand{\sT}{{\sf T}}

\newcommand\red[1]{\textcolor{black}{#1}}

\allowdisplaybreaks
\makeatletter
\newcommand{\thickhline}{
    \noalign {\ifnum 0=`}\fi \hrule height 1pt
    \futurelet \reserved@a \@xhline
}
\newcolumntype{"}{@{\hskip\tabcolsep\vrule width 1pt\hskip\tabcolsep}}
\makeatother
\begin{document}
\title[Stationary measures of CTMCs with applications to SRNs]
  {Stationary measures of continuous time Markov chains with applications to stochastic reaction networks}
\author{Mads Christian Hansen $^{1}$}
\author{Carsten Wiuf $^{1}$}
\author{Chuang Xu $^{2}$}
\email{wiuf@math.ku.dk (Corresponding author)}
\address{$^1$
Department of Mathematical Sciences,
University of Copenhagen, 2100 Copenhagen, Denmark.}
\address{$^2$
Department of Mathematics\\
University of Hawai'i at M\={a}noa, Honolulu\\
96822, HI, US.}

\date{\today}

\noindent

\begin{abstract} 
We study continuous-time Markov chains on the non-negative integers under mild regularity conditions (in particular, the set of jump vectors is finite and both forward and backward jumps are possible). Based on the so-called flux balance equation, we derive an iterative formula for calculating stationary measures.  Specifically, a stationary measure $\pi(x)$ evaluated at $x\in\N_0$ is  represented as a linear combination of  a few generating terms, similarly to the characterization of a stationary measure of a birth-death process, where there is only one generating term, $\pi(0)$. The coefficients of the linear combination are recursively determined in terms of the transition rates of the Markov chain.  For the class of Markov chains we consider, there is always at least one stationary measure (up to a scaling constant). We give various results pertaining to uniqueness and non-uniqueness of stationary measures, and show that the dimension of the linear space of signed invariant measures is at most the number of generating terms. A minimization problem is constructed in order to compute  stationary  measures numerically. Moreover, a heuristic linear approximation scheme is suggested for the same purpose   by first approximating  the generating terms. The correctness of the linear approximation scheme is justified in some special cases. Furthermore, a decomposition of the state space into different types of states (open and closed irreducible classes, and trapping, escaping and neutral states) is presented. Applications to stochastic reaction networks  are well illustrated.
\end{abstract}


\maketitle

\section{Introduction}
Stochastic models of interacting particle systems are often formulated in terms of continuous-time Markov chains (CTMC) on a discrete state space. These models find application in population genetics, epidemiology  and biochemistry, where the particles are known as  \emph{species}.  A natural and accessible framework for representing   interactions between species is  a  stochastic reaction network,  where  the underlying graph captures the possible jumps between states and the interactions between species. In the case where the reaction network consists of a single species, it is referred to as a \emph{one-species reaction network}. Such networks frequently arise in various applications, ranging from SIS models in epidemiology \cite{PCMV15}   to  bursty chemical processes (for example, in gene regulation) \cite{SS08} and the Schl\"{o}gl model \cite{FMD17}. 
One often focuses on examining the long-term dynamic behavior of the system, which can be captured by its corresponding limiting or stationary distribution, provided it exists. Therefore, characterizing the structure of such distributions is of great interest.

Stochastic models of reaction networks, in general, are highly non-linear, posing challenges for analytical approaches. Indeed, the characterization of  stationary distributions remain largely incomplete \cite{W86}, except for specific cases such as mono-molecular (linear) reaction networks \cite{JH07}, complex balanced reaction networks \cite{ACK10}, and when the associated stochastic process is a birth-death process \cite{A91}. To obtain statistical information, it is common to resort to stochastic simulation algorithms \cite{G77, ECM07}, running the Markov chain numerous times. However, this approach can be computationally intensive, rendering the analysis of complex reaction networks infeasible \cite{CM14}. Furthermore, it fails to provide analytical insights into the system.

We investigate one-species reaction networks  on the non-negative integers $\mathbb{N}_0$, and  present an analytic characterization of   stationary measures for general continuous-time Markov chains, subject to mild and natural regularity conditions (in particular, the set of jump vectors is finite and both forward and backwards jumps are possible), see Proposition \ref{pro-2}.  Furthermore, we provide a decomposition of the state space into different types of states: neutral, trapping, and escaping states, and positive and quasi-irreducible components  (Proposition \ref{pro-1old}). Based on this characterization, we provide an iterative formula to calculate $\pi(x)$, $x\in\N_0$, for any stationary measure $\pi$, not limited to stationary distributions, in terms of a few generating terms \red{(Theorem \ref{th-1})}; similarly to the characterization of the stationary distribution/measure of a birth-death process with one generating term $\pi(0)$. The iterative formula is derived from the general flow balance equation \cite{Ke11}. 

Moreover, we show that the linear subspace of signed invariant measures has dimension at most the number of generating terms and that each signed invariant measure is given by the iterative formula and a vector of generating terms \red{(Theorem \ref{th:posneg})}. We use  \cite{H57} to argue there always exists a stationary measure and give conditions for uniqueness and non-uniqueness \red{(Corollary \ref{cor:fraction}, Corollary \ref{Cor:uniqueness}, Theorem \ref{th:upward}, Lemma \ref{le:stationary-measure})}. Furthermore, we demonstrate by example that there might be two or more linearly independent stationary measures \red{(Example \ref{ex:nonunique})}.  As birth-death processes has a single generating term, then there cannot be a signed invariant measure taking values with both signs.

Finally, we demonstrate how the iterative formula can be used to approximate a stationary measure. Two methods are discussed: convex optimization \red{(Theorem \ref{th:scheme1})} and a heuristic linear approximation scheme \red{(Lemma \ref{lem:scheme3})}. We establish conditions under which the linear approximation scheme is correct, and   provide simulation-based illustrations to support the findings. Furthermore, we observe that even when the aforementioned conditions are not met, the linear approximation scheme still produces satisfactory results. In particular, it   allows us to recover stationary measures in certain cases. This  approach offers an alternative to truncation approaches \cite{GMK17, KTSB19} and  forward-in-time simulation techniques such as Gillespie's algorithm \cite{G77}.

\section{Preliminaries}

\subsection{Notation}

Let  $\R_{\ge 0}$, $\R_{> 0}$, $\R$ be the  non-negative real numbers,  the positive real numbers, and the real numbers, respectively, $\Z$ the integers, $\N$ the natural numbers and $\N_0=\N\cup\{0\}$ the non-negative integers.
For $m, n\in\N_0$, let $\R^{m\times n}$ denote the set of $m\times n$ matrices over $\R$.  The sign of $x\in\R$ is  defined as $\text{sgn}(x)=1$ if $x>0$, $\text{sgn}(x)=0$ if $x=0$, and $\text{sgn}(x)=-1$ if $x<0$. We use $\lc \cdot\rc$ to denote the ceiling function, $\lf \cdot\rf$ to denote the floor function, and  $||\cdot||_p$ to denote the $p$-norm. Furthermore, let $\mathbbm{1}_B\colon D\to \{0,1\}$ denote the indicator function of a set $B\subseteq D$, where $D$ is the domain.

\subsection{Markov Chains}

We define a class of CTMCs on $\N_0$ in terms of a finite set of jump vectors and  non-negative transition functions. The setting can be extended to CTMCs on $\N^d_0$ for $d>1$ and to infinite sets of jump vectors.  Let $\cT\subseteq\Z\!\setminus\!\{0\}$ be a finite set and $\cF=\lt\{\lambda_{\om}\colon \om\in\cT\rt\}$ a set of non-negative transition functions on $\N_0$,
$$\lambda_{\om}\colon\N_0\to\R_{\ge0},\quad \om\in\cT.$$
The notation is standard in reaction network literature \cite{AK15}, where we find our primary source of applications.
For convenience, we let
\begin{equation}\label{eq:lambda0}
\lambda_\omega(k)=0\quad \text{for}\quad k<0, \quad\text{and}\quad\lambda_\omega\equiv 0\quad \text{for}\quad\omega\not\in\Omega.
\end{equation}

The transition functions define a $Q$-matrix $Q=(q_{x,y})_{x,y\in \N_0}$ with  $q_{x,y}=\lambda_{y-x}(x)$, $x,y\in\N_0$, and subsequently, a class of CTMCs $(Y_t)_{t\ge 0}$ on $\N_0$ by assigning an initial state $Y_0\in\N_0$.
Since $\Omega$ is finite, there are at most finitely many  jumps from any $x\in\N_0$.
For convenience, we identify the class of CTMCs with $(\Omega,\cF)$.

 A state $y\in\N_0$ is {\em reachable} from $x\in\N_0$
 if there exists a sequence of states $x^{(0)},\ldots,x^{(m)}$, such that $x=x^{(0)}$, $y=x^{(m)}$ and $\lambda_{\omega^{(i)}}(x^{(i)})>0$ with $\omega^{(i)}=x^{(i+1)}-x^{(i)}\in\cT$,  $i=0,\ldots,m-1$.
It is \emph{accessible} if it is reachable from some other state. The state is \emph{absorbing} if  no other states can be reached from it.
 A set $A\subseteq\N_0$ is a \emph{communicating class} of $(\Omega,\cF)$ if any two states in $A$ can be reached from one another, and the set cannot be extended while preserving this property.
  A state $x\in\N_0$ is a {\em neutral state} of $(\Omega,\cF)$ if it is an absorbing state not accessible from any other state,
a {\em trapping state} of $(\Omega,\cF)$ if it is an absorbing state accessible from some other state,
and an {\em escaping state} of $(\Omega,\cF)$ if  it forms its own communicating class and some other state is accessible from it.
A set $A\subseteq\N_0$ is a {\em positive irreducible component} (\emph{PIC}) of $(\Omega,\cF)$ if it is a non-singleton closed  communicating class, and  a {\em quasi-irreducible component} (\emph{QIC}) of $(\Omega,\cF)$ if it is  a non-singleton  open communicating class.

Let $\sN$, $\sT$, $\sE$, $\sP$, and $\sQ$ be the (possibly empty) set of all neutral states, trapping states, escaping states, PICs and QICs of $(\Omega,\cF)$, respectively. Each state has a unique type, hence $\sN$, $\sT$, $\sE$, $\sP$, and $\sQ$ form a decomposition of the state space into disjoint sets.

A \emph{non-zero} measure $\pi$ on a closed irreducible component $A\subseteq \N_0$ of $(\Omega,\cF)$ is a  \emph{stationary measure} of $(\Omega,\cF)$ if $\pi$   is invariant for the $Q$-matrix, that is, if $\pi$ is a  non-negative equilibrium of the  {\em master equation} \cite{G92}:
 \begin{equation}\label{eq:master}
 0=\sum_{\omega\in\cT}\lambda_{\omega}(x-\omega)\pi(x-\omega)-\sum_{\omega\in\cT}\lambda_{\omega}(x)\pi(x),\quad x\in A,
 \end{equation}
and a \emph{stationary distribution} if it is a stationary measure and $\sum_{x\in A}\pi(x)=1$. Furthermore, we say  $\pi$  is \emph{unique}, if it is  \emph{unique up to a scaling constant}. We might leave out `on $A$' and just say $\pi$ is stationary measure of $(\Omega,\cF)$, when the context is clear.

\subsection{Stochastic reaction networks (SRNs)}
 
 \red{For clarity, we only introduce one-species SRNs and  not  SRNs in generality \cite{AK15}, as one-species  SRNs  are our primary application area.  A one-species SRN is a finite  labelled digraph   $(\mathcal{C},\mathcal{R})$ with an associated CTMC on $\N_0$. The elements of $\mathcal{R}$ are   \emph{reactions}, denoted by $y\ce{->[\eta]} y'$, where $y,y'\in\mathcal{C}$ are   \emph{complexes} and the label is a non-negative intensity  function on $\N_0$. Each complex is   an integer multiple   of the species $\tS$, that is, $n\tS$ for some $n$.   In examples, we simply draw the reactions as  in the following example with  $\mathcal{C}=\{0, \tS, 2\tS, 3\tS\}$, and four reactions:}
\begin{equation}\label{eq:example}
4\tS\ce{->[\eta_1]}  2\tS \ce{<=>[\eta_2][\eta_3]} 0 \ce{->[\eta_4]}6\tS.
\end{equation} 

\red{For convenience, we represent   the complexes as elements of $\N_0$ via the natural embedding, $n\tS\mapsto n$, and number  the reactions as in the example above.  The associated stochastic process $(X_t)_{ t\geq 0}$ can be given as
\begin{align*}
X_t=X_0+\sum_{k=1}^{\# \mathcal{R}}(y'_k-y_k) Y_k\left(\int_0^t \eta_k(X_s)\,ds\right),
\end{align*}
where $Y_k$ are independent  unit-rate Poisson processes and $\eta_k\colon\mathbb{N}_0\to[0,\infty)$ are intensity functions \cite{AK15,EK,markov}. 
 By varying the initial  species count  $X_0$, a whole family of Markov chains is associated with the SRN.
}

A common choice of intensity functions is that of stochastic \emph{mass-action kinetics} \cite{AK15}, 
\begin{align*}
\eta_k(x)=\kappa_k \frac{x!}{(x-y_k)!}, 
\qquad x\in\N_0,
\end{align*}
\red{where $\kappa_k>0$ is the {\em reaction rate constant}  of the $k$th reaction, and the combinatorial factor is the number of ways $y_k$ molecules can be chosen out of $x$ molecules (with order).}   This intensity function satisfies the stoichiometric admissibility condition:
\begin{align*}
 \eta_k(x) > 0 \quad \Leftrightarrow \quad x\geq  y_{k}
\end{align*}
($\ge$ is taken component-wise).
\red{Thus, every reaction can only fire when the copy-numbers of the  species in the current state are no fewer than those of the  source complex}.

We will assume   mass-action kinetics in many examples below and label the reactions with their corresponding reaction rate constants, rather than the intensity functions.
To bring SRNs into the setting of the previous section, we define
$$\red{\Omega=\{y'_k-y_k\,|\, y_k \to  y_k'\in\mathcal{R}\}},$$
$$\red{\lambda_\om(x)=\sum_{  k=1}^{\#\mathcal{R}}1_{\{\omega\}}(y_k'-y_k)\eta_k(x),\quad x\in\N_0,\quad\om\in\Omega.}$$
  
\section{A classification result}

In this section, we classify the state space $\N_0$ into different types of states. In particular, we are interested in characterising the PICs in connection with studying stationary measures. Our primary goal is to understand the class of one-species SRNs, which we study by introducing   a larger class of Markov chains embracing SRNs.

We assume throughout that a class of CTMCs  associated with $(\cT,\cF)$ is given.   Define
\begin{align*}
\Omega_+:=\left\{\omega\in \Omega\colon \text{sgn}(\omega)=1\right\},\quad \Omega_-:=\left\{\omega\in \Omega\colon \text{sgn}(\omega)=-1\right\}
\end{align*}
as the sets of positive and negative jumps, respectively.
\red{To avoid trivialities and for regularity, we make the following  assumptions.}

\medskip
\noindent
{\bf Assumptions.}
\begin{itemize}
\item[({\bf A1})] $\Omega$ is finite, $\cT_-\neq\varnothing$ and $\cT_+\neq\varnothing$.  
\item[({\bf A2})] \red{For $\om\in\Omega$, there exists $\si_\om\in\N_0$ such that $\lambda_\om(x)>0$ if and only if $x\ge \si_\om$}
\end{itemize}

\medskip\noindent
\red{Then, $\si_\om$ is the smallest state for which a jump of size $\om$ can occur ($\si$ is for input). If either of $\Omega_-$ and $\Omega_+$ is empty, then the chain is either a pure death or a pure birth process.  Assumption  ({\bf A1}) and ({\bf A2}) are fulfilled for stochastic mass-action kinetics. }

\red{For the classification, we need some further  quantities. Let $\so_\omega:=\si_\omega+\omega$ be the smallest possible `output' state after applying a jump of size $\om$, and let}
 $$ \si:=\text{min}_{\om\in\cT}\si_{\om}, \quad \si_+:=\text{min}_{\om\in\cT_+}\si_{\om},\quad \so:=\text{min}_{\om\in\cT}\so_{\om} \quad\so_-:=\text{min}_{\om\in\cT_-}\so_{\om}.$$
 \red{Any state $x< \si$ is a trapping or neutral  state  (no jumps can occur from one of these states),   and $\si_+$ is the smallest state for which a forward jump can occur. Similarly, $\so$ is the smallest state that can be reached from any other state and $\so_-$, the smallest state that can be reached by a backward jump. }

\red{In example \eqref{eq:example}, all jumps are   multiples of $2$, that is, the Markov chain is on $2\N_0$ or $2\N_0+1$, depending on the initial state. Generally, the number of infinitely large irreducible classes  is $\om_*=\text{gcd}(\Omega)$, the greatest  \emph{positive} common divisor of $\om\in\Omega$ ($\om_*=2$ in example \eqref{eq:example}).} The following classification result is a consequence of   \cite[Corollary 3.15]{XHW20b}.

\begin{proposition}\label{pro-1old}
Assume $\emph{({\bf A1})}$-$\emph{({\bf A2})}$. Then
$$\sN=\{ 0,\ldots,\min\{\si,\so\}-1\},\quad \sT=\{\so,\ldots,\si-1\},\quad \sE=\{\si,\ldots,\max\{\si_+,\so_-\}-1\}.$$
Furthermore, the  following hold:
\begin{enumerate}
\item If $\#\sT=0$, then $\sQ=\varnothing$, and  $\sP_s=\om_*\N_0+s$,   $s = \so_-,\ldots,\so_-+\om_*-1$, are the PICs,
\item If  $\#\sT\ge \omega_*$,  then  $\sP=\varnothing$, and  $\sQ_s=\om_*\N_0+s$, $s =\si_+,\ldots,\si_++\om_*-1$,  are the QICs,
\item If $0<\#\sT<\omega_*$, then $\sP_s=\om_*\N_0$ $+s$, $s=\si_+,\ldots,\so_-+\om_*-1$, are the PICs, and
$\sQ_s=\om_*\N_0+s$, $s=\so_-+\om_*,\ldots,\si_++\om_* -1$, are the QICs.
.
\end{enumerate}
In either case, there are $\om_*$ PICs and QICs in total. When PICs exist,  these are indexed by $s = \max\{\si_+,\so_-\},\ldots,\so_-+\om_*-1$, and   $\sP\not=\emptyset$ if and only if  $\si_+<\so_-+\om_*$.
\end{proposition}

\begin{proof}
We apply \cite[Corollary~3.15]{XHW20b}.    
To translate the notation of the current paper to that of \cite[Corollary~3.15]{XHW20b}, we put $c=0$, $L_c=\N_0$, $c_*=\max\{\si_+,\so_-\}$, $\om^*=\om_*$, and $\omega^{**}=\omega_*$. Then, the expressions of $\sN$, $\sT$, and $\sE$ naturally follow. As in \cite{XHW20b}, define the following sets,
   $$\Sigma^+_0=\left\{1+v-\left\lfloor\frac{v}{\omega_*}\right\rfloor\omega_*\colon v\in\sT\right\},$$
   $$\Sigma^-_0=\left\{1+v-\left\lfloor\frac{v}{\omega^*}\right\rfloor\omega_*\colon v\in\{\si,\ldots,\so+\om_*-1\}\setminus\sT\right\}.$$
    Since for $v\in\N_0$,
  \[1\le 1+v-\left\lfloor\frac{v}{\omega_*}\right\rfloor\omega_*\equiv 1+v\,\,\text{mod}\,\,\omega_* <1+\omega_*,\]
  we have
$\Sigma_0^+\cap\Sigma_0^-=\emptyset,$ $\Sigma^+_0\cup\Sigma^-_0=\{1,\ldots,\omega_*\}$,
and $\#\Sigma^+_0=\min\{\#\sT,\omega^*\}$. If $\so< \si$, these conclusions follow easily; if $\so\ge\si$, then $\Sigma_0^+=\emptyset$, and the conclusions follow.

  According to \cite[Corollary~3.15]{XHW20b}, it follows that
  \begin{equation}\label{eq:PQ}
\omega_*\left(\N_0+\left\lceil\frac{c_*-(v-1)}{\omega_*}\right\rceil\right)+v-1=\begin{cases}
    \sP^{(v)}_0,\quad v\in\Sigma^-_0,\\
    \sQ^{(v)}_0,\quad v\in\Sigma^+_0,
  \end{cases}
\end{equation}
are      the disjoint PICs and    the disjoint QICs of $(\Omega,\cF)$, respectively, in the terminology of  \cite{XHW20b}. Consequently,  as $\#(\sN\cup\sT\cup\sE)=c_*$,
 \[\bigcup_{v\in\Sigma^-_0\cup\Sigma^+_0}\left( \sP^{(v)}_0\cup \sQ^{(v)}_0\right)=\N_0\setminus(\sN\cup\sT\cup\sE)=\N_0+c_*.\]
Since, for $v\in\Z$,
  \[ c_*\le \omega_*\left\lceil\frac{c_*-(v-1)}{\omega_*}\right\rceil+v-1< \omega_*+c_*,\]
then   we might state \eqref{eq:PQ} as
  \begin{align*}
 \sP^{(v)}_0&=(\omega_*\Z+v-1)\cap(\N_0+c_*),\quad \text{for} \quad v\in\Sigma^-_0, \\
   \sQ^{(v)}_0&=(\omega_*\Z+v-1)\cap(\N_0+c_*),\quad \text{for} \quad v\in\Sigma^+_0.
  \end{align*}

We  show that the expressions given for $\sP^{(v)}_0,\sQ^{(v)}_0$ correspond to those given for $\sP_s,\sQ_s$ in the three cases, by suitable renaming of the indices.
First, note that $\sT=\varnothing$ if and only if $\so\ge\si$.
 From $\so\le\so_-<\si_-:=\text{min}_{\om\in\cT_-}\si_{\om}$ and $\so\ge\si$,   we have $\si=\si_+\le\so_-$, which yields   $c_*=\so_-$. Consequently, $\Sigma_0^+=\emptyset$ and $\sQ=\varnothing$. This proves the expression for  $\sP_s$ in (1).

  Otherwise, if $\so<\si$, then $\so<\si\le\si_+<\so_+:=\text{min}_{\om\in\cT_+}\so_{\om}$, which implies that $\so=\so_-<\si_+$. Hence, $c_*=\si_+$. If $\#\sT\ge\om_*$, then $\sP=\varnothing$, which proves  the expression for $\sQ_s$ in (2).
It remains to prove the last case. If $0<\#\sT<\omega_*$, then
 \begin{align*}
 \sP^{(v+1)}_0&=(\omega_*\Z+v)\cap(\N_0+c_*),\quad \text{for} \quad v=\si,\ldots,\so+\omega_*-1.
    \end{align*}
If $\si=\si_+$, then using the above equation and  $\so=\so_-$, the expression for  $\sP_s$ in (3) follows directly,  and the remaining irreducible classes must be QICs. Finally, we show $\si<\si_+$ is impossible, which concludes the proof. Assume oppositely that $\si<\si_+$. Then, $\si=\si_-$, $\sT=\{\so_-,\ldots,\si_--1\}$, and $\si_-\in \sE$. This implies one can jump from the state $\si_-$ (smallest state for which a backward jump can be made) leftwards to a state $x\le \si_--\omega_*< \so_-$. The latter inequality comes from $0<\#\sT=\si-\so<\omega_*$ and $\so=\so_-$. However, this implies $x\in\sN$, which is impossible.

The total number of PICs and QICs follows from $\Sigma^+_0\cup\Sigma^-_0=\{1,\ldots,\omega_*\}$. The indexation follows  from $c_*=\max\{\si_+,\so_-\}$ in the two case (1) and (3). Also, the inequality $\si_+<\so_-+\om_*$ follows straightforwardly in these two cases. It remains to check that it is not fulfilled in case (2). In that case, $\#\sT=\si-\so\ge\omega_*$ by assumption, hence $\si_+\ge\si\ge \so+\om_*=\so_-+\om_*$, and the conclusion follows.
\end{proof}

In either of the three cases of the proposition, the index $s$ is   the smallest element of the corresponding  class (PIC or QIC). \red{The role of Assumption ({\bf A2}) in Proposition \ref{pro-1old}, together with Assumption ({\bf A1}), is to ensure the non-singleton irreducible classes are infinitely large. If the assumption fails there could be non-singleton finite irreducible classes, even when $\Omega_+\neq\emptyset$.}

 A stationary distribution exists trivially on  each element of $\sN\cup\sT$. If the chain starts in $\sE$, it will eventually be trapped into a closed class, either into $\sT$ or $\sP$, unless absorption is not certain in which case it might be trapped in $\sQ$. If absorption is certain it will eventually leave $\sQ$.  We give two  examples of CTMCs on $\N_0$ showing how the chain might be trapped.

\begin{example}
(A) The reaction network	$\tS\ce{->[]}0,\ 2\tS\ce{<=>[]}3\tS$
has $\Omega=\{1,-1\}$, $\omega_*=1$, $\si_1=2$, $\si_{-1}=1$ (note that there are two reactions with jump size $-1$). Hence $\si=1$, $\si_+=2$ and $\so=\so_-=0$.		It follows from Proposition~\ref{pro-1old} that $\sN=\varnothing$, $\sT=\{0\}$, $\sE=\{1\}$, $\sP=\varnothing$, and $\sQ=\N_0+2$. There is only one infinite class, which is a QIC. From the escaping state, one can only jump backward to the trapping state. The escaping state is reached from $\sQ$.

(B) The reaction network	$\tS\ce{->[]}2\tS,\  2\tS\ce{<=>[]}3\tS$
has $\Omega=\{1,-1\}$, $\omega_*=1$, $\si_1=1$,  $\si_{-1}=3$ (note that there are two reactions with jump size $1$). Hence, $\si=1$, $\si_+=1$ and $\so=\so_-=2$ . It follows  from Proposition~\ref{pro-1old} that $\sN=\{0\}$, $\sT=\varnothing$, $\sE=\{1\}$, $\sP=\N_0+2$ and $\sQ=\varnothing$. There is only one infinite class, which is a PIC. From the escaping state, one can only jump forward to this PIC.
\end{example}

\section{Characterization of stationary measures}\label{sec:charac}

We present an  identity for stationary measures based on the flux balance equation \cite{Ke11}, see also \cite{XHW20c}. It  provides means to express any term of a stationary measure  as a linear combination with real coefficients of the   generating terms. The coefficients are determined recursively, provided ({\bf A1}) and ({\bf A2}) are fulfilled. 

To smooth the presentation we assume without loss of generality  that
\begin{itemize}
\item[({\bf A3})]    $\omega_*=1$, $s=0$ and   $\sP_0=\N_0$.
\end{itemize}
Hence, we assume $\N_0$ is a PIC and  remove the index $s$ from the notation for convenience. For general $(\Omega,\cF)$ with $\omega_*\ge1$ and $s\in\{\max\{\si_+,\so_-\},\ldots,\so_-+\om_*-1\}$, we translate the Markov chain to one fulfilling ({\bf A3}) for each $s$ by defining $(\Omega_*,\cF_s)$ by $\Omega_*=\Omega \omega_*^{-1}$ (element-wise multiplication) and $\cF_s=\{\lambda^s_{\omega}\colon \omega\in\Omega_*\}$ with $\lambda^s_\omega(x)=\lambda_{\omega\omega_*}( \omega_*x+s)$, $x\in\N_0$.  Hence, there are no loss in assuming ({\bf A3}).   \red{We assume $({\bf A1})$-$({\bf A3})$ throughout Section \ref{sec:charac} unless otherwise stated.}

\red{Let $\pi$ be any stationary measure of $(\Omega,\cF)$  on  $ \N_0 $. Define $\om_+,\om_-$ to be the largest positive and negative jump sizes, respectively,
\begin{equation}\label{eq:plusminus}
\om_+=\max\,\Omega_+,\quad \om_-=\min\, \Omega_-.
\end{equation}
We will show that   $\pi(x)$ can be expressed as a linear combination with real coefficients of the \emph{generating terms} $\pi(\sL),\ldots,\pi(\sU)$, where
\begin{equation}\label{eq:om=1}
\sL=\si_{\om_-}+\om_-=\so_{\om_-},\qquad \sU=\si_{\om_-}-1.
\end{equation}
 are the \emph{lower} and \emph{upper} numbers, respectively.  Hence, as a sum of $\sU-\sL+1=-\om_-$ terms. No backward jumps of size $\om_-$ can occur from any state $x\le \sU$, and $\sL$ is the smallest output state of a jump of size $\om_-$.}  
 
\begin{example}\label{ex:LUs}
(A) Consider the reaction network, $0\ce{<=>[][]}2\tS,$ $5\tS\ce{->[]}\tS.$ In this case, $\omega_*=2$,  $\si_+=\so_-=0$, and ({\bf A3}) does not apply. In fact,  $s=0,1$ with PICs $2\N_0$ and $2\N_0+1$, respectively. After translation, we  find $\sL_0=1$, $\sU_0=2$, and $\sL_1=0$, $\sU_1=1$, where the index refers to $s=0,1$. Thus, the lower and upper numbers are not the same for each class.
(B) The reaction network, $0\ce{<=>[][]}\tS,$ $n\tS\ce{<=>[]}(n+2)\tS,$
has  $\omega_*=1$,  $\si_+=\so_-=0$. Hence, $s=0$ with PIC $\N_0$, and ({\bf A3}) applies.  We find $\sL=n,\sU=n+1$. Thus, the lower and upper numbers might be  arbitrarily large depending on the SRN.
\end{example} 

Before presenting the main results, we study an example.

\begin{example}  \label{ex:equations}
Recall Example \ref{ex:LUs}(B) with mass-action kinetics, $n=2$, and reactions 
$$0\ce{<=>[\ka_1][\ka_2]}\tS\quad 2\tS\ce{<=>[\ka_3][\ka_4]}4\tS.$$
According to \cite[Theorem 7]{HWX}, this SRN is positive recurrent on $\N_0$ for  all $\ka_1,\ldots,\ka_4>0$ and hence, it has a unique stationary distribution.
We have $\sL=2$, $\sU=3$. 

By rewriting the master equation \eqref{eq:master}, we obtain
\begin{align*}
\pi(0)& =\frac{\kappa_2}{\kappa_1}\pi(1),\quad \pi(1)=\frac{2\kappa_2}{\kappa_1}\pi(2),\\
  \pi(4)&=\sum_{i=1}^3\frac{\eta_i( 2)}{\eta_4(4)}\pi( 2)-\frac{\eta_1(1)}{\eta_4(4)} \pi(1)-\frac{\eta_2(3)}{\eta_4(4)} \pi(3),\\
 \pi(\ell)&=\sum_{i=1}^4\frac{\eta_i(\ell-2)}{\eta_4(\ell)}\pi(\ell-2)-\frac{\eta_1(\ell-3)}{\eta_4(\ell )} \pi(\ell-3)-\frac{\eta_2(\ell-1)}{\eta_4(\ell )} \pi(\ell-1)-\frac{\eta_3(\ell-4)}{\eta_4(\ell )} \pi(\ell-4),
\end{align*}
the latter for $\ell> \sU+1=4$. There is not an equation   that expresses $\pi(3)$ in terms of the lower states $\ell<3$ as the state $3$ can only be reached from $2, 4$ and $5$. Consequently,  $\pi(0)$ and $\pi(1)$ can be found from     $\pi(2)$ and $\pi(3)$, but not vice versa, and   $\pi(\ell)$, $\ell\ge\sU+1=4$ is determined recursively   from $\pi(2)$ and $\pi(3)$ using the last equations above, say, $\pi(\ell)=\gamma_2(\ell)\pi(2)+\gamma_3(\ell)\pi(3)$, where the coefficients $\gamma_2(\ell), \gamma_3(\ell)$ depend  on the intensity functions. For $\ell=0,1$, these follow from the first equations (see also Theorem \ref{th-1} below),
\begin{align*}
 \gamma_2(0)&=\frac{2\kappa_2^2}{\kappa_1^2},\quad \gamma_3(0)=0,\quad \gamma_2(1)=\frac{2\kappa_2}{\kappa_1},\quad \gamma_{3}(1)=0,
 \end{align*}
while for $\ell=2,3$, we obviously have $\gamma_j(\ell)=\delta_{j,\ell}$, where $\delta_{\ell,j}$ is the Kronecker symbol. For $\ell>3$, the coefficients are given recursively, see   Theorem \ref{th-1} for the general expression. A recursion for $\pi(\ell)$ cannot be given in terms of $\pi(0),\ldots,\pi(1)$ alone.  
\end{example}

\color{black}

First, we focus on the real coefficients of the linear combination. \red{From Example \ref{ex:equations} we learn that the coefficients take different forms depending on the state, which is reflected in the definition of $\gamma_j(\ell)$ below.} For convenience, rows and columns  of a matrix are indexed from $0$.

\red{By the definition  of $\omega_+$ and $\omega_-$, any state is reachable from at most $\om_+-\om_-$ other states in one jump. For this reason, let} 
\begin{equation}\label{eq:mstar}
m_*=\om_+-\om_--1\ge 1
\end{equation}
(as $\om_+\ge 1$, $-\om_-\ge 1$), and define the functions 
$$\gamma_j\colon\Z\to\R,\quad \sL\leq j\leq \sU,$$
 by
\begin{equation}\label{Eq-1}
\gamma_j(\ell)=\left\{\begin{array}{cl}
0 &  \quad\text{for}\quad \ell<0, \\
 G_{\ell,j-\sL}  & \quad\text{for}\quad \ell= 0,\ldots,\sL-1, \quad j < \sU, \\
0 & \quad\text{for}\quad  \ell= 0,\ldots,\sL-1,\quad j=\sU, \\
\delta_{\ell,j}& \quad\text{for}\quad \ell= \sL.\ldots,\sU, \\
\sum_{k=1}^{m_*}\gamma_j(\ell-k)f_k(\ell) & \quad\text{for}\quad \ell>\sU,
\end{array}\right.
\end{equation}
where  the functions $f_k\colon [\sU+1,\infty)\cap\N_0\to\R$ are  defined by
\begin{align}\label{eq:f}
f_k(\ell)&=-\frac{c_k(\ell-k)}{c_0(\ell)}\qquad \text{for}\qquad \ell>\sU,\quad k=0,\ldots,m_*,\\
c_k(\ell)&={\rm sgn}(\om_-+k+1/2)\sum_{\omega\in B_k}\lambda_{\om}(\ell),\quad \ell\in\Z,\quad k=0,\ldots,m_*,\label{eq:c} \\
B_k&=\{\omega\in\Omega\ |\ k'(\om-k')>0,\,\,\,\text{with}\,\,\, k'=\om_-+k+1/2\}\label{eq:Bk}
\end{align}
(note that $f_0(\ell)=-1$ is   not used in the definition of $\gamma_j(\ell)$, but appears in the proof of the Proposition \ref{pro-2} below),
and $G=-(H_1)^{-1}H_2$ is an $L\times (U-L)$ matrix defined from the $\sL\times \sU$ matrix $H=(H_1\,\,H_2)$ with entries
\begin{align}\label{eq:H}
H_{m,n}=\delta_{m,n}-\frac{\lambda_{(m-n)}(n)}{\sum_{\omega\in \Omega}\lambda_{\omega}(m)},\quad m=0,\ldots,\sL-1,\quad n=0,\ldots,\sU-1,
\end{align}
where $H_1$ is $\sL\times \sL$ dimensional and $H_2$ is $\sL\times (\sU-\sL)$ dimensional. Note that $H$ is the empty matrix if $\sL=0$ or $\sU=\sL$.
In \eqref{eq:H}, we adopt the convention stated in \eqref{eq:lambda0}. In particular, $H_{m,m}=1$ for $0\le m\le \sL-1$. By definition,   $(I_{\sL} \,\, -G)$ is the row reduced echelon form of $H$ by Gaussian elimination, where $I_\sL$ is the $\sL\times \sL$ identity matrix.

\red{The functions $c_k\colon\Z\to\R$ and the sets $B_k$ come out of the flux balance equation, that is an equivalent formulation of the master equation \cite{Ke11}. We use it in \eqref{eq:flux} in the proof of Proposition \ref{pro-2}. For each $x\in\N_0$, it provides an identity between two positive sums, one over terms evaluated in at most $\om_+$ states with values $\ell< x$, and one over terms evaluated in at most $-\om_-+1$ states with values $\ell\ge x$. }
The function $f_k(\ell)$ is  well defined for $\ell>\sU$ if ({\bf A1}), ({\bf A2}) are fulfilled. In that  case, $c_0(\ell)<0$, see Lemma \ref{lem:ck}. The matrix $H$ is well defined with invertible $H_1$ under the assumptions ({\bf A1}), ({\bf A2}), provided $\sL>0$ and $\sU>\sL$, see Lemma \ref{lem:echelon}.

Proposition \ref{pro-2} and Theorem~\ref{th-1} below do not require uniqueness of the stationary measure.  

\begin{proposition}\label{pro-2} 
A  non-zero measure $\pi$ on $\N_0$ is a stationary measure of $(\Omega,\cF)$ on $\N_0$ if and only if
\begin{equation} \label{eq:mstar2original}
\sum_{k=0}^{m_*}\pi(\ell-k)c_k(\ell-k)=0,\quad  \text{for}\quad \ell>\sU-\sL.
\end{equation}
Here, $\pi(\ell)=0$  for $\ell<0$ for convenience,   $c_k$ is defined as in \eqref{eq:c},  and $m_*$ is defined as in \eqref{eq:mstar}.
\end{proposition}

\begin{proof}
We recall an identity related to the flux balance equation \cite{Ke11} for stationary measures for a CTMC on the non-negative integers \cite[Theorem~3.3]{XHW20c}, which is a  consequence of the master equation \eqref{eq:master}: A non-negative measure (probability measure)
$\pi$ on $\N_0$ is a stationary measure (stationary distribution) of $(\Omega,\cF)$ if and only if
\begin{equation}\label{eq:flux}
-\sum_{j=\omega_-+1}^0{\pi(x-j)}\sum_{\omega\in A_j}\lambda_{\omega}(x-j)+\sum_{j=1}^{\omega_+}\pi(x-j)\sum_{\omega\in A_j}\lambda_{\omega}(x-j)=0,\quad x\in\Z,
\end{equation}
where it is used that $\#\Omega<\infty$, and  the domain of $\pi$ as well as $\lambda_{\omega}$ is extended from $\N_0$ to $\Z$ (that is, $\pi(x)=0$, $\lambda_{\omega}(x)=0$ for $x\in\Z\setminus\N_0$). Furthermore,  the sets $A_j$ are defined by
$$A_j=\begin{cases}
\{\omega\in\cT\colon j>\omega\}\quad \text{if}\ j\in\{\omega_-,\ldots,0\},\\
\{\omega\in\cT\colon j\le\omega\}\quad \text{if}\ j\in\{1,\ldots,\omega_++1\}.
\end{cases}$$
If $x\le 0$ then all terms in both double sums in \eqref{eq:flux} are zero. In fact, \eqref{eq:flux} for $x$ is equivalent to the master equation \eqref{eq:master} for $x-1$.

Assume $\pi$ is a stationary measure of $(\Omega,\cF)$. Let $x=\ell+(\omega_-+1)$,  $\ell\in\Z$, and  let  $j=k+(\omega_-+1)$ with $j\in\{\omega_-+1,\ldots,\omega_+\}$. Then,   $x-j=\ell-k$, and it follows that \eqref{eq:flux} is equivalent to
\begin{equation}\label{eq:flux2}
\sum_{k=0}^{m_*}\textrm{sgn}(\omega_-+k+1/2) \pi(\ell-k)\sum_{\omega\in A_{k+\omega_-+1}}\lambda_{\omega}(\ell-k)=0,
\end{equation}
where $\textrm{sgn}(\omega_-+k+1/2)=1$ if $\omega_-+k\ge0$, and $-1$ otherwise.
It implies  that
\[(k+\omega_-+1)>\omega\quad\Leftrightarrow\quad(k+\omega_-+1/2)>\omega,\]
\[{(k+\omega_-+1)\le \omega}\quad\Leftrightarrow\quad(k+\omega_-+1/2)<\omega.\]
Hence, it also follows from the definition \eqref{eq:Bk} of $B_k$ that
\[A_{k+\omega_-+1}=B_k,\quad k=0,\ldots,m_*.\]
Thus, \eqref{eq:flux2} is equivalent to
\begin{equation}\label{eq:mstar}
\sum_{k=0}^{m_*}\pi(\ell-k)c_k(\ell-k)=0,\qquad  \ell\in\N_0,
\end{equation}
using the definition of $c_k(\ell)$. Since $\ell=x-(\om_-+1)=x+\sU-\sL$ and \eqref{eq:flux} is trivial for $x\le 0$, then \eqref{eq:mstar} is also trivial for $\ell=0,\ldots,\sU-\sL$ (all terms are zero).
This proves the bi-implication and the proof is completed.
\end{proof}

For $0\le\ell\le\sU-\sL$, all $m_*$ terms in \eqref{eq:mstar2original} are zero, hence the identity is a triviality.
We express  $\pi$ in terms of the generating terms,  $\pi(\sL),\dots,\pi(\sU)$. 

\begin{theorem}\label{th-1}
A  non-zero measure $\pi$ on $\N_0$ is a stationary measure of $(\Omega,\cF)$ on $\N_0$ if and only if   
\begin{equation}\label{Eq-3}
\pi(\ell)=\sum_{j=\sL}^{\sU}\pi(j)\gamma_j(\ell), \quad\text{for}\quad \ell\in\N_0, 
\end{equation}
where $\gamma_j$ is defined as in \eqref{Eq-1}.
\end{theorem}

\begin{proof}
Assume $\pi$ is a stationary measure. The proof is by induction. We first prove the induction step.
Assume \eqref{Eq-3} holds for $\ell-1\ge \sU$ and all $\ell'$ such that $\ell-1\ge \ell'\ge0$. Then, from  Proposition~\ref{pro-2},  \eqref{Eq-1}, \eqref{eq:f}, \eqref{eq:c}, and the induction hypothesis, we have
$$\pi(\ell)=-\sum_{k=1}^{m_*}\pi(\ell-k)\frac{c_k(\ell-k)}{c_0(\ell)} =-\sum_{k=1}^{m_*}\sum_{j=\sL}^{\sU}\pi(j)\gamma_j(\ell-k)\frac{c_k(\ell-k)}{c_0(\ell)}=\sum_{j=\sL}^{\sU}\pi(j)\gamma_j(\ell).$$

We next turn to the induction basis.
For $\ell=\sL,\ldots,\sU$,  \eqref{Eq-3} holds trivially as $\gamma_j(\ell)=\delta_{\ell,j}$ in this case.
It remains to prove  \eqref{Eq-3} for $\ell=0,\ldots, \sL-1$. Since $\pi$ is a stationary measure it fulfils the master equation \eqref{eq:master} for  $(\Omega,\cF)$. By rewriting this, we obtain
 \[\pi(\ell)=\frac{\sum_{\om\in\cT}\lambda_{\om}(\ell-\om)\pi(\ell-{\om})}{\sum_{\om\in\cT}\lambda_{\om}(\ell)},\qquad \ell\in\N_0.\]
The denominator is never zero because for $\ell\ge 0$, it holds that  $\lambda_\om(\ell)>0$ for at least one $\om\in\Omega$ (otherwise zero is a trapping state).

In particular, for  $\ell=0,\ldots,\sL-1$, using \eqref{eq:om=1}, it holds that  $\ell-\om_-<\si_{\om_-}$. Hence, $\lambda_{\om_-}(\ell-\om_-)=0$, and
\[\pi(\ell)=\frac{\sum_{\om\in\cT\setminus\{\om_-\}}\lambda_{\om}(\ell-\om)\pi(\ell-\om)}{\sum_{\om\in\cT}
\lambda_{\om}(\ell)},\qquad \ell=0,\ldots,\sL-1.\]
Now, defining $n=\ell-\om$, we have  $n< \sL-1-\om_-=\sU$ for $\om\in\cT\setminus\{\om_-\}$, using    $\sU-\sL=-(\om_-+1)$, see  \eqref{eq:om=1}. Hence
$$\pi(\ell)=\frac{\sum_{n=0}^{\sU-1}\lambda_{\ell-n}(n)\pi(n)}{\sum_{\om\in\cT}
\lambda_{\om}(\ell)},\qquad \ell=0,\ldots,\sL-1,$$
with the convention in \eqref{eq:lambda0}.
Evoking the definition of $H$ in \eqref{eq:H}, this equation may be written as
\[H\begin{pmatrix}
  \pi(0)\\
  \vdots\\
  \pi(\sU-1)
\end{pmatrix}=0.
\]
Recall that  $G=-(H_1)^{-1}H_2$ with $H=(H_1\,\,H_2)$. Noting that $\gamma_{\sU}(\ell)=0$, $\ell=0,\ldots,\sL-1$, yields that \eqref{Eq-3} is fulfilled with $\gamma_j(\ell)=G_{\ell,j-\sL}$, $\ell=0,\ldots,\sL-1$ and $j=\sL,\ldots,\sU-1$, and the proof of the first part is concluded.

For the reverse part, we note that for $0\le x-1\le \sL-1$, the argument above is `if and only if': $\pi(x-1)=\sum_{j=\sL}^\sU \pi(j)\gamma_j(x-1)$ if and only if the master equation \eqref{eq:master} is fulfilled for $x-1$. As noted in the proof of Proposition \ref{pro-2}, this is equivalent to \eqref{eq:flux} being fulfilled for $x$, which in turn is equivalent to \eqref{eq:mstar} being fulfilled for $\ell=x+\sU-\sL$ (the equation is replicated here),
\begin{equation} \label{eq:mstar2}
\sum_{k=0}^{m_*}\pi(\ell-k)c_k(\ell-k)=0.
\end{equation}
As $0\le x-1\le \sL-1$, then \eqref{eq:mstar2} holds for $\ell=\sU-\sL+1,\ldots,\sU$.

For $\ell>U$  we have, using  \eqref{Eq-1} and \eqref{eq:f},
\begin{align*}
\sum_{k=0}^{m_*}\pi(\ell-k)c_k(\ell-k)&=\sum_{k=0}^{m_*}\sum_{j=\sL}^{\sU}\pi(j)\gamma_j(\ell-k)c_k(\ell-k)=\sum_{j=\sL}^{\sU}\pi(j) \sum_{k=0}^{m_*}\gamma_j(\ell-k)c_k(\ell-k)\\
&=\sum_{j=\sL}^{\sU} \pi(j)\gamma_j(\ell)c_0(\ell)+\sum_{j=\sL}^{\sU}\pi(j) \sum_{k=1}^{m_*}\gamma_j(\ell-k)c_k(\ell-k)\\
&=0,
\end{align*}
hence \eqref{eq:mstar2} holds for all $\ell>\sU-\sL$. According to Proposition \ref{pro-2}, $\pi$ is then a stationary measure of $(\Omega,\mathcal{F})$.
\end{proof}


For completeness, we apply Proposition \ref{pro-2} to Example \ref{ex:equations}. The SRN has $\om_-=-2$ and $\om_+=2$, such that $m_*=\om_++-\om_--1=3$.   Equation   \eqref{eq:mstar2}  for $1=\sU-\sL< \ell \le\sU=3$  becomes
$$-\ka_2 \pi(1)+\ka_1\pi(0)=0,\quad -2\ka_2 \pi(2)+\ka_1\pi(1)=0,$$
in agreement with the equations found in Example \ref{ex:equations}.

\color{black}


\subsection{Matrix representation}

A simple representation can be obtained in terms of determinants of a certain  matrix. For  an infinite matrix $B$, let  $B(\ell)$, $\ell\ge 1$, denote the   matrix  composed of $B$'s first $\ell$ rows and columns.

\begin{theorem} \label{th-3}
The function $\gamma_j\colon\Z\to\R$, $\sL\le j\le \sU$, has the following determinant representation for $\ell>0$,
\begin{align*}
\gamma_j(\sU+\ell)=&\det B_j(\ell) 
\end{align*}
where the matrix $B_j$ is  the infinite band matrix,
{\small
\begin{align*}
B_{j}=\begin{pmatrix}
g_j(1) & -1 & 0 & 0 & \ldots &0&\ldots\\
g_j(2) & f_1(\sU+2) & -1 & 0 & \ldots&0&\ldots\\
\vdots & \vdots & \vdots & \vdots &\ddots&\vdots& \\
g_j(m_*) & f_{m_*-1}(\sU+m_*) & f_{m_*-2}(\sU+m_*) & f_{m_*-3}(\sU+m_*) &\ldots& -1&\ldots\\
0 & f_{m_*}(\sU+m_*+1) & f_{m_*-1}(\sU+m_*+1) & f_{m_*-2}(\sU+m_*+1)& \ldots&f_1(\sU+m_*+1)&\ldots\\
0 & 0 &  f_{m_*}(\sU+m_*+2) & f_{m_*-1}(\sU+m_*+2) &\ldots& f_2(\sU+m_*+2)&\ldots\\
0 & 0 & 0 & f_{m_*}(\sU+m_*+3) & \ldots&f_3(\sU+m_*+3)&\ldots\\
\vdots & \vdots & \vdots & \vdots & &\vdots &\ddots
\end{pmatrix},
\end{align*}}
 and
\begin{align*}
g_j(k)=\sum_{i=0}^{m_*-k}f_{m_*-i}(\sU+k)\gamma_j(\sU+i+k-m_*),\qquad k=1,\ldots,m_*.
\end{align*}
\end{theorem}

\begin{proof}
For the first equality, we apply  \cite[Theorem 2]{K93} on the solution to a $m_*$-th order linear difference equation with variable coefficients. First, from \eqref{Eq-1}, we have for $\ell>\sU$ that
\begin{align*}
\gamma_j(\ell)=\sum_{k=1}^{m_*}\gamma_j(\ell-k)f_k(\ell),\qquad\text{that is,}\qquad  \sum_{k=0}^{m_*}\gamma_j(\ell-k)f_k(\ell)=0,
\end{align*}
as $f_0(\ell)=-1$. Thus, for $\ell>\sU-m_*$, we have
\begin{align}\label{eq:used}
\sum_{k=0}^{m_*}\gamma_j(\ell+k)f_{m_*-k}(\ell+m_*)=0,
\end{align}
subject to the initial conditions $\gamma_j(\ell)$ for $\ell=\sU+1-m_*,\ldots,\sU$, given in \eqref{Eq-1}.

Secondly, we translate the notation used here into the notation of \cite[Theorem 2]{K93}: $n=m_*$, $i=k$, $k=\ell-(\sU-m_*)$. Furthermore, we take $g(k)=0$, $q(a,b)=f_{m_*-(b-a)}(\sU+a)$ and $y_{i+k}=\gamma_j(i+k+\sU-m_*)$.  Then,
the condition
$\sum_{i=0}^mq(k,i+k)y_{i+k}=g(k)$ of \cite[Theorem 2]{K93} translate into \eqref{eq:used}. The matrix $B_j$ is as in \cite[Theorem 2]{K93}, except the first column is multiplied by $-1$. By \cite[Theorem 2]{K93}, we find that
\begin{align*}
\gamma_j(\sU+\ell)=\frac{(-1)^{\ell-1}}{\prod_{i=1}^\ell f_0(\sU+i)}(-\det B_j(\ell))=\det B_j(\ell),\quad \text{for} \quad \ell\in\N.
\end{align*}
By Lemma \ref{lem:ck}, the second equality follows from $f_0(\ell)=-1$, $\ell>\sU$, see \eqref{eq:f}.
\end{proof}

A combinatorial representation can be achieved by calculating the determinant. In the present case, the matrix $B_j(\ell)$ is a Hessenberg matrix and an explicit combinatorial expression exists \cite[Proposition~2.4]{marrero}.
Combining Theorem \ref{th-1} and Theorem \ref{th-3} yields the following.

 \begin{corollary}\label{cor:B}
Let $\pi$ be a stationary measure of  $(\Omega,\cF)$  on $\N_0$. Then,
\begin{align*}
\pi(\ell)=\sum_{j=\sL}^{\sU}\pi(j)\det B_j(\ell-\sU), \qquad \ell>\sU.
\end{align*}
\end{corollary}

\begin{example}[Upwardly skip-free process]
 Consider  the SRN  with  mass-action kinetics,
$$0\ce{->[\ka_1]}\tS,\quad \tS\ce{->[\ka_2]}2\tS,\quad 2\tS\ce{->[\ka_3]}0.$$
Here, $\cT=\{-2,1\}$ and $\lambda_{-2}(x)=\ka_3x(x-1)$, $\lambda_1(x)=\ka_1+\ka_2x$ for $x\in\N_0$. Hence, $\om_*=1$, $\si_+=0$, $\so_-=0$. Consequently, $s=0$, $\sP_0=\N_0$ and ({\bf A3}) is fulfilled, in addition to ({\bf A1})-({\bf A2}). According to \cite[Theorem 7]{HWX}, the SRN is positive recurrent on $\N_0$ for all $\ka_1,\ka_2,\ka_3>0$, and hence has a unique stationary distribution.
 
We have $\sL=0$ and $\sU=1$.
By Theorem~\ref{th-3}, the stationary distribution fulfils 
\begin{align*}
\pi(x)=\pi(0) \det B_{0}(x-1)+\pi(1) \det B_{1}(x-1), \qquad x\geq2,
\end{align*}
where $B_0=( A_0\,\, B )$, $B_1=( A_1\,\, B )$, and
\[A_0=\begin{pmatrix}
\frac{\ka_1}{2\ka_3} \\
0 \\
0 \\
0 \\
0 \\
 \vdots \end{pmatrix},\quad
 A_1=\begin{pmatrix}
0 \\
\frac{\ka_1+\ka_2}{2\cdot3\ka_3} \\
0 \\
0 \\
0 \\
 \vdots \end{pmatrix},
 \quad B=\begin{pmatrix}
 -1 & 0 & 0 & 0&\ldots\\
 \frac{1}{3} & -1 & 0 & 0&\ldots\\
 \frac{\ka_1+2\ka_2}{3\cdot4\ka_3} & \frac{2}{4} & -1 & 0 &\ldots \\
 0 & \frac{\ka_1+3\ka_2}{4\cdot5\ka_3} & \frac{3}{5} & -1&\ldots\\
 0 & 0 & \frac{\ka_1+4\ka_2}{5\cdot6\ka_3} & \frac{4}{6}&\ldots\\
  \vdots &  \vdots&  \vdots& \vdots& \ddots
\end{pmatrix}.\]
\end{example}

\medskip
\begin{example}[Downwardly skip-free process]
Consider  the SRN with   mass-action kinetics,
$$0\ce{->[\ka_1]}2\tS,\quad \tS\ce{->[\ka_2]}0,\quad 2\tS\ce{->[\ka_3]}\tS.$$
Here,  $\cT=\{-1,2\}$ and $\lambda_{-1}(x)=\ka_2x+\ka_3x(x-1)$, $\lambda_2(x)=\ka_1$ for $x\in\N_0$. Hence, $\om_*=1$, $\si_+=0$, $\so_-=0$. Consequently, $s=0$, $\sP_0=\N_0$ and ({\bf A3}) is fulfilled, in addition to ({\bf A1})-({\bf A2}). According to \cite[Theorem 7]{HWX}, the SRN is positive recurrent on $\N_0$ for all $\ka_1,\ka_2,\ka_3>0$, and hence has a unique stationary distribution.

We have $\sL=\sU=0$.
By Theorem~\ref{th-3}, the stationary distribution fulfils 
\begin{align*}
\pi(x)=\pi(0)\det B_0(x), \qquad x\ge 1,
\end{align*}
where
\[
B_0=\begin{pmatrix}
\frac{\kappa_1}{\ka_2} & -1 & 0 & 0&\ldots\\
\frac{\kappa_1}{2(\ka_2+\ka_3)} & \frac{\kappa_1}{2(\ka_2+\ka_3)} & -1 & 0&\ldots\\
0 & \frac{\kappa_1}{3(\ka_2+2\ka_3)} & \frac{\kappa_1}{3(\ka_2+2\ka_3)} &-1&\ldots\\
0 & 0 & \frac{\kappa_1}{4(\ka_2+3\ka_3)} & \frac{\kappa_1}{4(\ka_2+3\ka_3)} &\ldots\\
\vdots & \vdots& \vdots& \vdots & \ddots
\end{pmatrix}.
\]
\end{example}

\section{Skip-free Markov chains}

  \subsection{Downwardly skip-free Markov chains}

An explicit iterative formula can be derived from Theorem \ref{th-1} for  downwardly skip-free Markov chains.

\begin{corollary}\label{cor-skipfree}
Assume $\omega_-=-1$, and let $\pi(0)>0$. Then, $\pi$ is a stationary measure of   $(\Omega,\cF)$ on $\N_0$ if and only if
\begin{equation*}
\pi(\ell)=\pi(0)\gamma_0(\ell),\quad \text{for}\quad\ell\in\N_0,
\end{equation*}
 where
\begin{equation*}
\gamma_0(0)=1,\qquad \gamma_0(\ell)=\sum_{k=1}^{\om_+}\gamma_0(\ell-k)f_k(\ell),\quad \ell>0,
\end{equation*}
and  $f_k$ is as defined in \eqref{eq:f}. Consequently, there exists  a unique stationary measure of  $(\Omega,\cF)$ on $\N_0$.  Furthermore, if $\pi$ is a stationary distribution, then
$\pi(0)=\lt(\sum_{\ell=0}^{\infty}\gamma_0(\ell)\rt)^{-1}$.  
\end{corollary}

\begin{proof}   Since $\om_-=-1$, then  $\sL=\sU$, see \eqref{eq:om=1}. Moreover, $\si_{\om_-}=1$ as otherwise the state zero could  not be reached. Hence, $\sL=\sU= \si_{\om_-}-1=0$ from \eqref{eq:om=1}. Consequently, $\pi(\ell)=\pi(0)\gamma_0(\ell)$, $\ell\in\N_0$, from Theorem \ref{th-1}. Furthermore, $m_*=\om_+-\om_--1=\om_+$ such that the expression for $\gamma_0(\ell)$, $\ell\in\N_0$, follows from \eqref{Eq-1}. 
Positivity of $\gamma(\ell)$, $\ell\in\N_0$, follows from Theorem \ref{th:posneg}.
If $\pi$ is a stationary distribution, then $1=\sum_{\ell=0}^\infty\pi(\ell)=\pi(0)\sum_{\ell=0}^\infty\gamma_0(\ell)$, and the conclusion follows. 
\end{proof}

Corollary \ref{cor-skipfree} leads to the classical birth-death process characterization. 

\begin{corollary}
Assume $\cT=\{-1,1\}$. A measure $\pi$ on $\N_0$ is a  stationary measure    of  $(\Omega,\mathcal{F})$  on $\N_0$  if and only if
$$\pi(\ell)=\pi(0)\prod_{j=1}^\ell\frac{\lambda_{1}(j-1)}{\lambda_{-1}(j)},\quad\text{for}\quad \ell>0,$$
where $\pi(0)=\lt(1+\sum_{\ell=1}^{\infty}\prod_{j=1}^\ell\frac{\lambda_{1}(j-1)}{\lambda_{-1}(j)}\rt)^{-1}$ in the case of a stationary distribution.
\end{corollary}

\begin{proof} In particular, the process is downwardly skip-free. Furthermore,  $\om_+=1$, such that for $\ell>0$, we have from Corollary \ref{cor-skipfree}, \eqref{eq:f} and \eqref{eq:c},
$$\gamma_0(\ell)=\sum_{k=1}^{\om_+}\gamma_0(\ell-k)f_k(\ell)=\gamma_0(\ell-1)f_1(\ell)=\prod_{j=1}^\ell f_1(j)=(-1)^{\ell}\frac{\prod_{j=1}^\ell c_1(j-1)}{\prod_{j=1}^\ell c_0(j)}.$$
By definition of $B_k$ and $c_k(\ell)$, $k=0,1$ ($m_*=1$), we have $B_0=\{-1\}$, $B_1=\{1\}$,
$$c_0(\ell)=-\lambda_{-1}(\ell),\quad c_1(\ell)=\lambda_{1}(\ell). $$
By insertion, this yields
$$\gamma_0(\ell)=\prod_{j=1}^\ell\frac{\lambda_{1}(j-1)}{\lambda_{-1}(j)},\qquad \ell>0,$$
and the statement follows from Corollary \ref{cor-skipfree}.
\end{proof}

\subsection{Upwardly skip-free Markov chains}\label{sec:skipup}

\red{In general, we are not able  to give conditions for when an upwardly skip-free Markov chain has a unique stationary measure  (up to a scalar)  and when it has  more than one or even  infinitely many  linearly independent ones. However, in a particular case,   if there is not a unique stationary measure, we establish all stationary measures as an explicit  one-parameter family of measures (Corollary \ref{cor:fraction}). If the transition functions are polynomial then we show a stationary measure is always unique (Corollary \ref{Cor:uniqueness}). Thus, we need non-polynomial transition rates for non-uniqueness and   give one such example in Example \ref{ex:nonunique}.}

\begin{proposition}\label{prop:up}
Assume $\omega_-=-2$ and $\omega_+=1$, hence $\sU=\sL+1$. 
A non-zero measure $\pi$ on $\N_0$ is a stationary measure of $(\Omega,\cF)$ on $\N_0$ if and only if
\begin{equation}\label{eq:phi}
h(x)(\phi(x+1)+1)=\phi(x)\phi(x+1),\quad x\ge\sU+2,
\end{equation}
where 
\begin{align}
\phi(x)&=\frac{\pi(x-1)}{\pi(x)}\frac{\lambda_{-1}(x-1)+\lambda_{-2}(x-1)}{\lambda_{-2}(x)},\nonumber \\
h(x) &=\frac{(\lambda_{-1}(x-1)+\lambda_{-2}(x-1))(\lambda_{-1}(x)+\lambda_{-2}(x))}{\lambda_1(x-1)\lambda_{-2}(x)},\label{eq:h}
\end{align}
for $x\ge \sU+2$, and 
 \begin{equation}\label{eq:smallx}
\pi(x)= \frac{\pi(x+1)(\lambda_{-1}(x+1)+\lambda_{-2}(x+1))+\pi(x+2)\lambda_{-2}(x+2) }{\lambda_1(x)},\quad 0\le x\le \sU,
\end{equation}
with  $\lambda_{-1}\equiv0$ for convenience if there are no jumps of size $-1$.
If this is the case, then
\begin{equation}\label{eq:piphi}
\pi(x)=\pi(\sU+1)\prod_{j=\sU+1}^{x-1}\frac{\lambda_{-1}(j)+\lambda_{-2}(j)}{\lambda_{-2}(j+1)\phi(j+1)}, \quad x\ge \sU+2,
 \end{equation}
\end{proposition}

\begin{proof}
Recall the master equation for a stationary measure $\pi$, in the form of \eqref{eq:flux2}:
\begin{equation} \label{Eq-pi}
\pi(x)(\lambda_{-1}(x)+\lambda_{-2}(x))+\pi(x+1)\lambda_{-2}(x+1) = \pi(x-1)\lambda_{1}(x-1),\quad x\ge1.
\end{equation}
Define $\phi$ and $h$ as in the statement for $x\ge\sU+2$ (if $x=\sU+1$, then $\phi(x), h(x)$ might be zero; if $x<\sU+1$, then division by zero occurs  as $\sU+1=\si_{\omega_-}$). Dividing $\pi(x+1)\lambda_{-2}(x+1)$ on both sides of \eqref{Eq-pi} yields 
\begin{align*}
\frac{\pi(x)}{\pi(x+1)}\frac{\lambda_{-1}(x)+\lambda_{-2}(x)}{\lambda_{-2}(x+1)}+1 =& \frac{\pi(x-1)}{\pi(x)}\frac{\lambda_{-1}(x-1)+\lambda_{-2}(x-1)}{\lambda_{-2}(x)}\\
&\cdot\frac{\pi(x)}{\pi(x+1)}\frac{\lambda_{-1}(x)+\lambda_{-2}(x)}{\lambda_{-2}(x+1)}\\
&\cdot\frac{\lambda_1(x-1)\lambda_{-2}(x)}{(\lambda_{-1}(x-1)+\lambda_{-2}(x-1))(\lambda_{-1}(x)+\lambda_{-2}(x))},
\end{align*}
and  the identity \eqref{eq:phi}   follows. By rewriting the master equation, then \eqref{eq:smallx} follows. The calculations can be done in reverse order yielding the bi-implication. Equation \eqref{eq:piphi} follows by induction.
\end{proof}

\begin{lemma}\label{lem:hrec}
Let  $h\colon [N,\infty)\cap\N_0\to(0,\infty)$  and $\phi\colon [N,\infty)\cap\N_0\to\R$ be functions with $N\in\N_0$, and  such that 
\begin{equation}\label{eq:recursionphi}
h(x)(\phi(x+1)+1)=\phi(x)\phi(x+1),\quad x\ge N.
\end{equation}
Let $\phi^*$ be a positive number. Then,  
 $\phi$ is a positive function    with $\phi(N)=\phi^*$, if and only if  
$$\Psi_2(N)\le \phi^*\le\Psi_1(N),$$
where
$\Psi_1(x)=\lim_{k\to\infty} \psi(x,2k-1)$, $\Psi_2(x)=\lim_{k\to\infty} \psi(x,2k)$, and $\psi(x,k)$ is determined recursively by 
\begin{equation}\label{Eq-recurrence}
\psi(x,k)=h(x)\left(1+\frac{1}{\psi(x+1,k-1)}\right), \quad x\ge N,\quad k\ge 1,
\end{equation} with $\psi(x,0)=h(x)$.
\end{lemma}
 
 \begin{proof}
Let $\psi(x,k)$ be as in the statement and $\phi$ a positive function fulfilling \eqref{eq:recursionphi}. 
Note that $\psi(x,k)$ is the $k$-th convergent  of a generalized continued fraction, hence $\psi(x,2k)$ is increasing in $k\ge0$ and $\psi(x,2k+1)$ is decreasing in $k\ge0$. Indeed,  this follows from Theorem 4 (monotonicity of even and odd convergents) of \cite{Khinchin}.   See also the proof of Corollary \ref{cor:fraction}.

By induction, we will show that
\begin{equation*}
\psi(x,2k)<\phi(x)< \psi(x,2k+1),\quad  x\ge N,\quad k\ge 0,
\end{equation*}
from which it   follows that
\begin{equation}\label{eq:claim}
\Psi_2(x)=\lim_{k\to\infty}\psi(x,2k)\le\phi(x)\le \Psi_1(x)=\lim_{k\to\infty}\psi(x,2k-1),\quad  x\ge N.
\end{equation}

 For the base case,  it follows from \eqref{eq:recursionphi} that $\psi(x,0)=h(x)<\phi(x)$ for $x\ge N$, and thus
 $$\psi(x,1)=h(x)\left(1+\frac{1}{\psi(x+1,0)}\right)>h(x)\left(1+\frac{1}{\phi(x+1)}\right)=\phi(x).$$
 For the induction step, assume $\psi(x,2k)<\phi(x)< \psi(x,2k+1)$ for $x\ge N$ and some $k\ge0$. Then, using \eqref{Eq-recurrence},
\begin{align*}
\psi(x,2k+2)&=h(x)\left(1+\frac{1}{\psi(x+1,2k+1)}\right)<h(x)\left(1+\frac{1}{\phi(x+1)}\right)=\phi(x),\\
\psi(x,2k+3)&=h(x)\left(1+\frac{1}{\psi(x+1,2k+2)}\right)>h(x)\left(1+\frac{1}{\phi(x+1)}\right)=\phi(x).
\end{align*}
If $\phi(N)=\phi^*$, then the first implication follows from \eqref{eq:claim}. For the reverse implication, assume $\Psi_2(N)\le\phi(N)\le \Psi_1(N).$  Note from \eqref{eq:recursionphi} that $\phi(x+1)$ is  positive only if $h(x)<\phi(x)$. We will show that   $\Psi_2(N)\le \phi(N)\le\Psi_1(N),$ implies  $\Psi_2(x)\le\phi(x)\le \Psi_1(x) $ for all $x\ge N$, hence also $h(x)=\psi(x,0)< \Psi_2(x)\le\phi(x)$ for all $x\ge N$, and we are done.  
Letting $k\to\infty$ in \eqref{Eq-recurrence} yields
\[\Psi_1(x)=h(x)\bigg(1+\frac{1}{\Psi_2(x+1)}\bigg),\quad \Psi_2(x)=h(x)\bigg(1+\frac{1}{\Psi_1(x+1)}\bigg),\quad x\ge N.\]
Hence, it follows from the induction hypothesis that
$$ \Psi_2(x+1)=\frac1{\frac{\Psi_1(x)}{h(x)}-1}\le \frac1{\frac{\phi(x)}{h(x)}-1}=\phi(x+1)\le \frac1{\frac{\Psi_2(x)}{h(x)}-1}=\Psi_1(x+1),$$ 
using \eqref{eq:recursionphi}, and the claim follows.  
\end{proof}

Below, we give a general condition for uniqueness and show that uniqueness does not always hold by example. In Section \ref{sec:signed}, we give concrete cases where uniqueness applies.

\begin{corollary}\label{cor:fraction}
Assume $\omega_-=-2$ and $\omega_+=1$.  Choose $\pi^*>0$ and $\Psi_2(\sU+2)\le \phi^*\le\Psi_1(\sU+2),$ where $\Psi_1,\Psi_2$ are as in Lemma \ref{lem:hrec} and $h$ as in \eqref{eq:h}. Let $\phi$ be a solution to \eqref{eq:recursionphi}  in Lemma \ref{lem:hrec} with $\phi(\sU+2)=\phi^*$. Then,  \eqref{eq:piphi} and \eqref{eq:smallx} define a stationary measure $\pi$ of $(\Omega,\cF)$ on $\N_0$  by setting $\pi(\sU+1)=\pi^*$.

The measure is unique  if and only if $\Psi_2(\sU+2)=\Psi_1(\sU+2)$, if and only if the following sum is divergent for $x=\sU+2$,
\begin{align*}
H(x)&=\sum_{n=0}^\infty\left( h(x+2n+1)\prod_{i=0}^n \frac{h(x+2i)}{h(x+2i+1)}+\prod_{i=0}^n \frac{h(x+2i+1)}{h(x+2i)}\right) \\
&= \sum_{n=0}^\infty\frac{(\lambda_{-1}(x+2n)+\lambda_{-2}(x+2n))(\lambda_{-1}(x-1)+\lambda_{-2}(x-1))}{\lambda_1(x+2n)\lambda_{-2}(x+2n+1)}\frac 1{Q_n(x)}  \\
&\quad+\sum_{n=0}^\infty \frac{\lambda_{-1}(x+2n+1)+\lambda_{-2}(x+2n+1)}{\lambda_{-1}(x-1)+\lambda_{-2}(x-1)} Q_n(x),
\end{align*}
where
$$Q_n(x)=\prod_{i=0}^n\frac{\lambda_1(x+2i-1)\lambda_{-2}(x+2i)}{\lambda_1(x+2i)\lambda_{-2}(x+2i+1)}.$$
\end{corollary}

\begin{proof}
The first part of the proof is an application of Proposition \ref{prop:up} and Lemma \ref{lem:hrec} with $N=\sU+2$. The last bi-implication follows by noting that $\psi(x,k)$ in Lemma \ref{lem:hrec} is the $k$-th convergent of a generalized continued fraction,
\begin{equation*}
b_0+\cfrac{c_1}{b_1+\cfrac{c_2}{b_2+\cfrac{c_3}{b_3+ \cdots\vphantom{\cfrac{1}{1}}  }}} = h(x)(1+\cfrac{1}{h(x+1)(1+\cfrac{1}{h(x+2)(1+\cfrac{1}{h(x+3)(1+ \cdots\vphantom{\cfrac{1}{1}}  }}}
\end{equation*}
that is,
$b_n=c_{n+1}=h(x+n)$, $n\ge 0$. By transformation, this generalized continued fraction is equivalent to a (standard) continued fraction,
\begin{equation*}
a_0+\cfrac{1}{a_1+\cfrac{1}{a_2+\cfrac{1}{a_3+ \cdots\vphantom{\cfrac{1}{1}}  }}}
\end{equation*}
with 
$$a_{2n}=h(x+2n+1)\prod_{i=0}^n \frac{h(x+2i)}{h(x+2i+1)},\quad a_{2n+1}=\prod_{i=0}^n \frac{h(x+2i+1)}{h(x+2i)},\quad n\ge 0.$$
By \cite[Theorem 10]{Khinchin}, the continued fraction converges if and only if $\sum_{n=0}a_n=\infty$, which proves the bi-implication, noting that by the first part of the proof it is sufficient to check $x=\sU+2$, and using the concrete form of $h(x) $ in \eqref{eq:h}.
\end{proof}

We give a concrete example of a non-unique Markov chain in  Example \ref{ex:nonunique}  in Section \ref{sec:examples}.

\begin{corollary}\label{Cor:uniqueness}
Assume $\om_+=1$, $\om_-=-2$,  and that  $\lambda_{\om}(x)$, $\om\in\cT$, is a polynomial for  large $x$. Then, there is a unique stationary measure of $(\Omega,\cF)$ on $\N_0$.
\end{corollary}

\begin{proof}
We prove that the first series in $H(x)$ in Corollary \ref{cor:fraction} diverges for   $x\ge\sU+2$, and hence $H(x)=\infty$. Let $m_1=\deg\lambda_{1}(x)$, $m_2=\deg\lambda_{-2}(x)$. We first provide asymptotics of $Q_n(x)$ for large $n$. By the Euler-Maclaurin formula with $x\ge\sU+2$,
\begin{align*}
Q_n(x)&=\exp\left(\sum_{i=0}^n \log\bigl(1-\tfrac{\lambda_1(x+2i-1)}{\lambda_1(x+2i)}\bigr)+\log\bigl(1-\tfrac{\lambda_{-2}(x+2i)}{\lambda_{-2}(x+2i+1)}\bigr)\right)\\
&= \exp\left(\sum_{i=0}^n \log\bigl(1-\tfrac{m_1}{x+2i}+O((x+2i)^{-2})\bigr)+\log\bigl(1-\tfrac{m_2}{x+2i}+O((x+2i)^{-2})\bigr) \right) \\
&=\exp\left(\sum_{i=0}^n -\tfrac{m_1}{x+2i}+O((x+2i)^{-2})-\tfrac{m_2}{x+2i}+O((x+2i)^{-2})\right) \\
&=\exp\Bigl(-m_1\log(x+2n)-m_2\log(x+2n)+O(1)\Bigr) \\
&=C(x,2n)(x+2n)^{-(m_1+m_2)},
  \end{align*}
  where  $0<C_1(x)<C(x,2n)<C_2(x)<\infty$ for all $n$, and $C_1(x), C_2(x)$ are positive constants  depending on $x$.  This implies that the terms in the first series of $H(x)$ are bounded uniformly in $n$ from below:
  \begin{align*}
 \liminf_{n\to\infty}\frac{\lambda_{-2}(x+2n)\lambda_{-2}(x-1)}{\lambda_1(x+2n)\lambda_{-2}(x+2n+1)}\frac 1{Q_n(x)}>0,
  \end{align*}
  and hence the first series in $H(x)$ diverges. 
\end{proof}

\section{Invariant vectors and existence of stationary measures}\label{sec:signed}

In the previous section, we  focused on   non-zero measures   on $\N_0$ and conditions to ensure   stationarity. 
  However, the requirement of a non-zero measure is not essential. In fact, the conditions of Proposition \ref{pro-2} and  Theorem \ref{th-1} characterize  equally  
  any real-valued sequence  that solves the master equation \eqref{eq:master}. Henceforth, we refer to the vector $(v_\sL,\ldots,v_\sU)\in\R^{\sU-\sL+1}$   as  a \emph{generator} and the sequence as an \emph{invariant vector}. This implies the  linear subspace in $\ell(\R)$ of invariant vectors   is  $(\sU-\sL+1)$-dimensional. Such vector might or might not be a \emph{signed invariant measure} depending on whether   the positive or the negative part  of the vector has finite one-norm.


We assume  ({\bf A1})-({\bf A3})  throughout Section \ref{sec:signed}.
 If the CTMC is recurrent (positive or null), then it is well known that there exists a unique stationary measure.
For transient CTMCs, including explosive chains, there also   exists a non-zero stationary measure in the setting considered.

\begin{proposition}\label{prop:exists}
  There exists a non-zero stationary measure of $(\Omega,\cF)$ on $\N_0$.
\end{proposition}

\begin{proof}
It follows from  \cite[Corollary]{H57}  and \cite[Theorem3.5.1]{markov},   noting that the set of states \red{ with non-zero transition rates to any given state $x\in \N_0$}    is finite, in fact $\le \#\Omega$. 
\end{proof}

\begin{lemma}\label{lem:nonzero}
Let $\pi$ be a non-zero   invariant measure of $(\Omega,\cF)$ on $\N_0$ such that $\pi(x)\ge0$ for all $x\in\N_0$. Then, $\pi(x)>0$ for all $x\in\N_0$. In fact, $\pi$ is a non-zero stationary measure of $(\Omega,\cF)$.
\end{lemma}

\begin{proof}
Assume $\pi(x)=0$. By rewriting the master equation \eqref{eq:master}, we obtain
\begin{align*}
 \pi(x) &= \frac{1}{\sum_{\om\in\Omega}\lambda_\om(x)}\sum_{\om\in\Omega}\lambda_\om(x-\om)\pi(x-\om)=0.
\end{align*}
Let $x$ be reachable from $y\in \N_0$  in $k$ steps. If $k=1$, then it follows from above that $y=x-\om$ for some $\om\in\Omega$ and $\pi(y)=0$. If $k>1$, then $\pi(y)=0$ by induction in the number  of steps. Since $\N_0$ is irreducible by assumption, then any state can be reached from any other state in a finite number of steps, and $\pi(x)=0$ for all $x\in\N_0$. However, this contradicts the measure  is non-zero. Hence, $\pi(x)>0$ for all $x\in\N_0$.
\end{proof}


\begin{theorem}\label{th:posneg}
The   sequences $ \gamma_\sL,\ldots,\gamma_\sU $ form a maximal set of  linearly  independent invariant vectors in $\ell(\R)$.   
Moreover, $\gamma_j$ has positive and negative terms for all $j=\sL,\ldots,\sU$ if and only if $\omega_-<-1$, that is, if and only if $\sL\not=\sU$.  If $\sL=\sU$, then $\gamma_\sL$ has all terms positive. In any case, $\gamma(n)=(\gamma_\sL(n),\ldots,\gamma_\sU(n))\not=0$, for   $n\in\N_0$.
\end{theorem}

\begin{proof}
The former part follows immediately from Theorem~\ref{th-1}, since $(\gamma_j(\sL),\ldots,\gamma_j(\sU))$ is the $(j-\sL+1)$-th unit vector of  $\R^{\sU-\sL+1}$,  for $j=\sL,\ldots,\sU$, by definition. The latter part follows from Lemma \ref{lem:nonzero} and the fact that $\gamma_j(j)=1$ and  $\gamma_j(\ell) = 0$ for $\ell\in\{\sL,\ldots,\sU\}\setminus\{j\}$.  If $\sL=\sU$, then the linear space 
is one-dimensional and positivity of all $\gamma_\sL(\ell)$, $\ell\in\N_0$, follows from Lemma \ref{lem:nonzero} and the existence of a stationary measure, see  Proposition \ref{prop:exists}. The equivalence between $\omega_-<1$ and $\sL\not=\sU$ follows from \eqref{eq:om=1}. If $\gamma(n)=0$, then $\pi(n)=0$ for any   invariant vector $\pi$, contradicting the existence of a stationary measure.
\end{proof}

\begin{theorem}\label{th:G}
Let $G\subseteq \R^{\sU-\sL+1}$ be the set of generators of stationary measures of $(\Omega,\cF)$ on $\N_0$. Then, $G\subseteq \R_{>0}^{\sU-\sL+1}$ is a positive convex cone. Let $G_0=G\cup\{(0,\ldots,0)\}$ be the set 
including the zero vector, which generates  the null measure. Then, $G_0$ is a closed set. The set $\{v\in G_0\colon  \|v\|_1=1\}$ is closed and convex.
\end{theorem}

\begin{proof}
Positivity follows from Lemma \ref{lem:nonzero}. It is straightforward to see that $G$ is a positive convex cone. 
Assume $v^{(m)}\in G_0$, $m\in\N_0$, and $v^{(m)}\to v=(v_\sL,\ldots,v_\sU)\not\in G_0$ as $m\to\infty$. Using Lemma \ref{lem:nonzero}, there exists $\ell\in\N_0$, such that $\sum_{j=\sL}^\sU v_j\gamma_j(\ell)<0$.  Then, there also exists $m\in\N_0$ such that $\sum_{j=\sL}^\sU v^{(m)}_j\gamma_j(\ell)<0$ with $v^{(m)}=(v^{(m)}_\sL,\ldots,v^{(m)}_\sU)$, contradicting that $v^{(m)}\in G_0$. Hence, $G_0$ is closed. The last statement is immediate from the previous.
\end{proof}

\red{Part of Theorem \ref{th:G} might be found in \cite[Theorem~1.4.6]{positivepart}.}  In general, we do not have uniqueness of a stationary measure of $(\Omega,\cF)$ on $\N_0$, unless in the case of recurrent CTMCs. 
For downwardly skip-free processes, we have $\sL=\sU=0$, hence the space of signed invariant measures is one-dimensional and uniqueness  follows, that is, there does not exist a proper signed invariant measure  taking values with both signs.   

We end the section with a few results on uniqueness of stationary measures. For this we need   the following lemma.

\begin{lemma}\label{lem:auxresult}
Let $v=(v_\sL,\ldots,v_\sU)\in\R_{\ge0}^{\sU-\sL+1}$   be a non-zero generator and assume $\om_+=1$. Let
$$\nu(\ell)=\sum_{j=\sL}^{\sU}v_j\gamma_j(\ell),\quad \ell\in\N_0.$$
If,   either
\begin{equation}\label{eq:c1}
\nu(\ell)\ge 0,\quad \ell=n-(\sU-\sL),\ldots,n, 
\end{equation}
for some $n\ge\sU-\sL$, or
\begin{equation}\label{eq:c2}
\nu(\ell)= 0,\quad \ell=n-(\sU-\sL-1),\ldots,n, 
\end{equation}
for some $n\ge\sU-\sL-1$, then
\begin{equation}\label{eq:c3}
\nu(\ell)\ge 0,\quad \ell=0,\ldots,n.
\end{equation}
\end{lemma}

\begin{proof} 
Since $\om_+=1$, then $m_*=\om_+-\om_--1=-\om_-=\sU-\sL+1$. We use $-\om_-$ rather than $\sU-\sL+1$ in the proof for convenience. From Lemma \ref{lem:ck}, we have $c_k(\ell)\le 0$ for $\ell\in\N_0$ and $0\le k<-\om_-$, and $c_{-\om_-}(\ell)> 0$ for $\ell\in\N_0$. The latter follows from $\Omega_+=\{\om_+\}$ and $\si_{\om_+}=0$ in this case. Furthermore,  since $\nu$ defines an invariant  vector, then  from \eqref{eq:mstar2original}, 
\begin{align}\label{eq:piLn}
0&=\sum_{k=0}^{-\om_--1} \nu(n-k)c_k(n-k) +\nu(n+\om_-)c_{-\om_-}(n+\om_-).
\end{align}

 Assume \eqref{eq:c1} holds.  If $n=\sU-\sL$, then there is nothing to show. Hence, assume $n\ge \sU-\sL+1$. By assumption the sum in \eqref{eq:piLn} is non-positive, while the last term is non-negative and $c_{-\om_-}(n+\om_-+1)> 0$ as $n+\om_-=n-(\sU-\sL+1)\ge0$.  Hence, $\nu(n+\om_-)\ge 0$.  Continue recursively for $n:=n-1$ until $n+\om_-=0$. Note that $\nu(\ell)=v_\ell\ge0$ for $\ell=\sL,\ldots,\sU$, in agreement with the conclusion.

Assume \eqref{eq:c2} holds. If $n=\sU-\sL-1$, then there is nothing to prove.
 For $n\ge\sU-\sL$, we aim to show \eqref{eq:c1} holds from which the conclusion follows.  By simplification of \eqref{eq:piLn},
\begin{align*}
-\nu(n-(\sU-\sL))c_{\sU-\sL}(n-(\sU-\sL))&=\nu(n+\om_-)c_{-\om_-}(n+\om_-)
\end{align*}
If $c_{\sU-\sL}(n-(\sU-\sL))<0$, then either   $\nu(n-(\sU-\sL))$ and $\nu(n+\om_-)$ take the same sign or are zero. Consequently, a) $\nu(\ell)\ge 0$ for all $\ell=n+\om_-,\ldots,n$, or b) $\nu(\ell)\le 0$ for all $\ell=n+\om_-,\ldots,n$. Similarly, if $c_{\sU-\sL}(n-(\sU-\sL))=0$, then $\nu(n+\om_-)=0$. Consequently, a) or b) holds in this case too. If a), then \eqref{eq:c3} holds. If b), then \eqref{eq:c3} holds with reverse inequality by applying an argument similar  to the non-negative case.  However, $\nu(\ell)=v_\ell\ge0$, $\ell=\sL,\ldots,\sU$, and at least one of them is strictly larger than zero, by assumption. Hence, the negative sign cannot apply and the claim holds.
\end{proof}

For $n\ge\sU$, define the $(\sU-\sL+1)\times(\sU-\sL+1)$ matrix by
\begin{equation*}
A(n) = \begin{pmatrix}
\frac{\gamma(n-(\sU-\sL-1))^T}{\|\gamma(n-(\sU-\sL-1))\|_1 } \smallskip \\
\vdots \smallskip \\
\frac{\gamma(n-1)^T}{\|\gamma(n-1)\|_1 } \smallskip \\
\frac{\gamma(n)^T}{\|\gamma(n )\|_1 } \smallskip \\
{\bf 1}^T
\end{pmatrix},
\end{equation*}
where $\gamma(\ell)=(\gamma_\sL(\ell),\ldots,\gamma_\sU( \ell))$, ${\bf 1}=(1,\ldots,1)$ and $^T$ denotes transpose.  This matrix is well-defined by Theorem~\ref{th:posneg}. The rows of  $A(n)$, except the last one, has $1$-norm one, and all entries are between $-1$ and $1$. 

\begin{theorem}\label{th:upward}
\red{Assume  there exists a strictly increasing subsequence $(n_k)_{k\in\N_0}$, such that   $A(n_k)\to A$ as $k\to\infty$, with $\det(A)\not=0$. }

 1) There is at most one stationary measure  
 of $(\Omega,\cF)$ on $\N_0$, say $\pi$, with the property 
\begin{equation}\label{eq:lim2}
\lim_{k\to\infty} \frac{\pi(n_k)}{\lVert \gamma(n_k)\rVert_1}= 0.
\end{equation}
 
2) If $\om_+=1$ and  $\pi$ is a unique  
stationary measure  of $(\Omega,\cF)$ on $\N_0$, then \eqref{eq:lim2} holds.
\end{theorem}

\begin{proof}
1) Let 
$$\sigma(n)=\left(\frac{\pi(n-(\sU-\sL-1))}{\|\gamma(n-(\sU-\sL-1))\|_1},\ldots,\frac{\pi(n)}{\|\gamma(n)\|_1}\right),$$
 and let $\pi_*=(\pi(\sL),\ldots,\pi(\sU))$ be the generator of the stationary measure $\pi$. We have
$$A(n_k) \pi_*=\begin{pmatrix} \sigma(n_k) \\ 1\end{pmatrix}\to A\pi_*=\begin{pmatrix} {\bf 0} \\ 1\end{pmatrix},$$
as $k\to\infty$,  where ${\bf 0}$ is the zero vector of length $\sU-\sL$. Since $A$ is invertible, then 
$$\pi_*=A^{-1}\begin{pmatrix} {\bf 0} \\ 1\end{pmatrix}.$$
 Consequently, as this holds for any stationary measure with the property \eqref{eq:lim2}, then $\pi$ is unique, up to a scalar.

2) According to Proposition \ref{pro-1} below, $A(n)$ is invertible and there is a unique non-negative  (component-wise) solution to 
$$A(n)v^{(n)}=\begin{pmatrix} {\bf 0} \\ 1\end{pmatrix},$$
 for all $n\ge \sU$.  It follows that
$$A v^{n_k} =(A-A(n_k))v^{(n_k)} +A(n_k)v^{n_k} =(A-A(n_k))v^{(n_k)}+\begin{pmatrix} {\bf 0} \\ 1\end{pmatrix}  \to \begin{pmatrix} {\bf 0} \\ 1\end{pmatrix}$$
as $k\to\infty$, since $\|v^{(n_k)}\|_1=1$. Define 
$$v=A^{-1}\begin{pmatrix} {\bf 0} \\ 1\end{pmatrix}, $$
then $v$ is non-negative and $\|v\|_1=1$, since $v^{(n_k)}$ is non-negative,    
$\|v^{(n_k)}\|_1=1$ and $v^{(n_k)}\to v$ as $k\to\infty$. We aim to show that $v$ is the generator of the unique stationary measure $\pi$. 

Define an invariant vector $\nu^k$, for each $k\in\N_0$, by
 $$\nu^k(\ell)=\sum_{j=\sL}^\sU v_j^{(n_k)} \gamma_j(\ell),\quad \ell\in\N_0.$$
 We have $\nu^k(\ell)=0$ for all $\ell=n_k-(\sU-\sL-1),\ldots,n_k$. Hence, by Lemma \ref{lem:auxresult}, $\nu^k(\ell)\ge0$ for    $\ell=0,\ldots,n_k$. Fix $\ell\in\N_0$. Then for all large $k$ such that $n_k\ge \ell$, we have $\nu^k(\ell)\ge0$ and
$$\nu^k(\ell)\to \nu(\ell)=\sum_{j=\sL}^{\sU}v_j\gamma_j(\ell),\quad\text{for}\quad k\to\infty.$$
  Hence, $\nu(\ell)\ge 0$ for all $\ell\in\N_0$. Consequently, using Lemma \ref{lem:nonzero}, $\nu$ is a stationary measure and by uniqueness, it holds  that $\nu=\pi$, up to a scaling constant.
\end{proof}

To state the next result, we introduce some notation.  Let 
$$I_j^+=\{n\in\N_0\colon \gamma_j(n)>0\},\quad I_j^-=\{n\in\N_0\colon \gamma_j(n)<0\},\quad j=\sL,\ldots,\sU.$$
If $\sL=\sU$, then $I^+_\sL=\N_0$ and $I^-_\sL=\emptyset$. If $\sL\not=\sU$, then by the definition of $\gamma_j$, $I^+_j\not=\emptyset$ and $I^-_j\not=\emptyset$ (Theorem \ref{th:posneg}).  In general, it follows from  Theorem~\ref{th-1}, Proposition~\ref{prop:exists}, and Lemma~\ref{lem:nonzero}, that  $I_j^-\subseteq \cup_{i\not=j} I_i^+$,  $\cup_{i=\sL}^\sU I_i^+=\N_0$ and $\cap_{i=\sL}^\sU I_i^-=\emptyset$.
In particular, for $\sU-\sL=1$, that is, $\omega_-=-2$, then  $I_\sU^-\subseteq I_\sL^+$ and $I_\sL^-\subseteq I_\sU^+$. Below, in the proof of Lemma \ref{le:stationary-measure}, we show these four sets are infinite. Hence, we may define
\begin{equation}\label{eq:infsup}
r_1=-\limsup_{n\in I_\sL^- } \frac{\gamma_\sL(n)}{\gamma_\sU(n)},\quad r_2=-\liminf_{n\in I_\sU^-}\frac{\gamma_\sL(n)}{\gamma_\sU(n)}.
\end{equation}

\begin{lemma}\label{le:stationary-measure} 
Assume $\omega_-=-2$, that is, $\sU=\sL+1$. It holds that $0< r_1\le r_2<\infty$. A non-zero measure $\pi$ is  a   stationary measure if and only if
\begin{equation}\label{range}
\frac{\pi(\sU)}{\pi(\sL)}\in[r_1,r_2]. 
\end{equation}

Furthermore, a  stationary measure $\pi$ is unique   if and only if $r_1 = r_2$, if and only if
\begin{equation}\label{eq:lim}
\lim_{n\in I_\sL^- \cup I_\sU^-} \frac{\pi(n)}{\lVert \gamma(n)\rVert_1}= 0.
\end{equation}
If this is the case, then limes superior and limes inferior in \eqref{eq:infsup} may be replaced by limes.
\end{lemma}

\begin{proof}
Let $\pi$ be a stationary measure, which exists by Proposition~\ref{prop:exists}. Then $\pi(n) = \pi(\sL)\gamma_\sL(n) +\pi(\sU)\gamma_\sU(n)>0$, which implies that 
\[\frac{\pi(\sU)}{\pi(\sL)}>-\frac{\gamma_\sL(n)}{\gamma_\sU(n)},\quad \forall n\in I_\sL^-,\quad\text{and}\quad \frac{\pi(\sU)}{\pi(\sL)}<-\frac{\gamma_\sL(n)}{\gamma_\sU(n)},\quad \forall n\in I_\sU^-.\]
by taking supremum  and infimum, respectively, this further implies    
$$\widetilde r_1=\sup\left\{-\frac{\gamma_\sL(n)}{\gamma_\sU(n)}\colon n\in I_\sL^-\right\}\le \inf\left\{-\frac{\gamma_\sL(n)}{\gamma_\sU(n)}\colon n\in I_\sU^-\right\}=\widetilde r_2,$$ 
and that $\pi(\sU)/\pi(\sL)$ is  in the interval $[\widetilde r_1,\widetilde r_2]$.  
Note that $\widetilde r_2<\infty$ and $\widetilde r_1>0$, since $I_\sU^-$ and $I_\sL^-$ are both non-empty, using Theorem \ref{th:posneg}. For the reverse conclusion, for any   invariant vector $\pi$ such that \eqref{range} holds, we have using Theorem~\ref{th-1},
 $$\pi(n) = \pi(\sL) \gamma_\sL(n) + \pi(\sU) \gamma_\sU(n)\ge0,\quad \forall n\in\N_0,$$
  which implies that $\pi$ is a stationary measure, see Lemma \ref{lem:nonzero}.

Assume that either $\widetilde r_1$ or $\widetilde r_2$ is attainable for some $n\in\N_0$. Then, there exists  a stationary measure $\pi$ such that 
$$\frac{\pi(\sU)}{\pi(\sL)}=-\frac{\gamma_\sL(n)}{\gamma_\sU(n)},$$
that is, $\pi(n) = \pi(\sL)\gamma_\sL(n)+\pi(\sU)\gamma_\sU(n) = 0$, contradicting   the positivity of $\pi$, see  Lemma~\ref{lem:nonzero}.  Hence, $I_\sL^-$ and $I_\sU^-$ both contain infinitely many elements, and since neither $\widetilde r_1$ nor $\widetilde r_2$ are attainable, then $\sup$ and $\inf$ can be replaced by $\limsup$ and $\liminf$, respectively, to obtain $r_1$ and $r_2$ in \eqref{eq:infsup}. Hence, also  $I_\sL^+$ and $I_\sU^+$ are infinite sets, since $I_\sU^-\subseteq I_\sL^+$ and $I_\sL^-\subseteq I_\sU^+$.

For the second part, the bi-implication with $r_1=r_2$ is straightforward. Assume $\pi$ is a stationary measure, such that $\pi(\sL),\pi(\sU)>0$. Note that
$$g_1=\liminf_{n\in I^-_\sL} \frac{\gamma_\sU(n)}{\lVert \gamma(n)\rVert_1}>0,\quad g_2=\liminf_{n\in I^-_\sU} \frac{\gamma_\sL(n)}{\lVert \gamma(n)\rVert_1}>0,$$
as otherwise $\pi(n) = \pi(\sL) \gamma_\sL(n) + \pi(\sU) \gamma_\sU(n)<0$ for some large $n$.
 For $n\in I^-_\sL$,
\begin{align}\label{eq:om2}
\frac{\pi(n)}{\lVert \gamma(n)\rVert_1} &= \pi(\sL)\frac{\gamma_\sL(n)}{\lVert \gamma(n)\rVert_1}+\pi(\sU)\frac{\gamma_\sU(n)}{\lVert \gamma(n)\rVert_1}=\left( \frac{\gamma_\sL(n)}{\gamma_\sU(n)}+\frac{\pi(\sU)}{\pi(\sL)}\right)\frac{\pi(\sL)\gamma_\sU(n)}{\lVert \gamma(n)\rVert_1}\\
&\ge \pi(\sL)(g_1-\epsilon)\left( \frac{\gamma_\sL(n)}{\gamma_\sU(n)}+\frac{\pi(\sU)}{\pi(\sL)}\right)\ge 0, \nonumber
\end{align}
for some small $\epsilon>0$ and large $n$. Similarly, for $n\in I^-_\sU$,
\begin{align}\label{eq:om3}
\frac{\pi(n)}{\lVert \gamma(n)\rVert_1} &= \pi(\sL)\frac{\gamma_\sL(n)}{\lVert \gamma(n)\rVert_1}+\pi(\sU)\frac{\gamma_\sU(n)}{\lVert \gamma(n)\rVert_1}=\left(\frac{\pi(\sL)}{\pi(\sU)}+ \frac{\gamma_\sU(n)}{\gamma_\sL(n)}\right)\frac{\pi(\sU)\gamma_\sL(n)}{\lVert \gamma(n)\rVert_1}\\
&\ge \pi(\sU)(g_2-\epsilon)\left(\frac{\pi(\sL)}{\pi(\sU)}+ \frac{\gamma_\sU(n)}{\gamma_\sL(n)}\right) \ge 0,\nonumber
\end{align}
for $\epsilon>0$ and large $n$. Taking $\limsup$ over $n\in I^-_\sL$ in \eqref{eq:om2} and  $\limsup$ over $n\in I^-_\sU$  in \eqref{eq:om3}, using \eqref{eq:lim}, yields 
$$r_1=-\limsup_{n\in I^-_\sL} \frac{\gamma_\sL(n)}{\gamma_\sU(n)}=\frac{\pi(\sU)}{\pi(\sL)},\quad \frac1{r_2}=-\limsup_{n\in I^-_\sU} \frac{\gamma_\sU(n)}{\gamma_\sL(n)}=\frac{\pi(\sL)}{\pi(\sU)},$$
or with $\limsup$ replaced by $\liminf$. Hence, \eqref{eq:infsup} holds with $\limsup$ and $\liminf$ replaced by $\lim$.
It follows that $r_1=r_2$ and that $\pi$ is unique. For the reverse statement, change the first inequality sign $\ge$ in \eqref{eq:om2} and \eqref{eq:om3} to $\le$, and $(g_1-\epsilon)$ and $(g_2-\epsilon)$ to one, and the conclusion follows.  
\end{proof}

\section{Convex constrained optimization }\label{sec:convex}

When the Markov chain is sufficiently complex,   an analytical expression for a stationary measure may not exist. In fact, this seems to be the rare case. If an analytical expression is not available, one may find estimates of the relative sizes of $\pi(\sL), \ldots \pi(\sU)$, which in turn determines $\pi(\ell)$, $\ell\ge0$, up to a  scaling constant, by Theorem \ref{th-1}. If $\pi$ is a stationary distribution, this constant may then be found numerically by calculating the first $n$ probabilities $\pi(\ell)$, $\ell=0,\ldots,n$, for some $n$, and renormalizing to obtain a proper distribution. Here, we examine how one may proceed in practice.

\begin{theorem}\label{th:scheme1}
Assume $\emph{({\bf A1})}$-$\emph{({\bf A3})}$ and for $n\ge0$, let
\begin{equation}\label{eq:Kn}
K_n=\left\{v\in\R^{\sU-\sL+1}_{\ge0}\colon \sum_{j=\sL}^{\sU}v_j\gamma_j(\ell)\geq 0, \  \ell=0,\ldots,n,\ \lVert v\rVert_1= 1\right\}.
\end{equation}
Then,  $K_n\not=\emptyset$, $K_n\supseteq K_{n+1}$, $n\ge0$, and  
$\cap_{n=1}^{\infty}K_n\subseteq\R^{\sU-\sL+1}_{>0}$ is non-empty and consists of all generators of non-zero stationary measures of $(\Omega,\cF)$ on $\N_0$, appropriately normalized. In fact, $G_0=\cap_{n=1}^{\infty}K_n$ with $G_0$ as in Theorem \ref{th:G}.

Furthermore, there is a unique  minimizer  $v^{(n)}$ to the following constraint quadratic optimization problem,
\begin{equation}\label{Eq-4}
\min_{v\in K_n} \lVert v\rVert_2^2.
\end{equation}
Moreover, the limit  
 $v^*=\lim_{n\to\infty}v^{(n)}\in \cap_{n=1}^{\infty}K_n$ exists and equals 
the generator of a stationary measure $\pi$   of $(\Omega,\mathcal{F})$ on $\N_0$,  that is, $v^*=(\pi(\sL),\ldots, \pi(\sU))$. 
\end{theorem}

\begin{proof}
The sets $K_n$ are non-empty by Proposition \ref{prop:exists} and  obviously non-increasing: $K_n\supseteq K_{n+1}$.  Hence, since all $K_n$s have a common element, the intersection $\cap_{n=1}^{\infty}K_n$ is also non-empty. Lemma \ref{lem:nonzero} rules out the intersection  contains boundary elements of $\R^{\sU-\sL+1}_{\ge0}$.

Furthermore, the sets $K_n$ are non-empty, closed and convex. Since $\|\cdot\|_2^2$ is  strictly convex, there exists a unique minimizer $v^{(n)}\in K_n$ for the constrained optimization problem \eqref{Eq-4}  \cite[p137 or Ex 8.1 on p.447]{boyd2004convex}.

Since  the sets $K_n$ are   non-increasing, then $v^{(n)}$ is non-decreasing in $\ell_2$-norm: $\lVert v^{(n)} \rVert_2^2\le \lVert v^{(n+1)} \rVert_2^2$ for all $n\ge1$. Furthermore, since the sets 
$K_n$ are  closed subsets of the simplex and hence compact, any sequence $ v^{(n)}$, $n\ge1$, of minimizers has a converging subsequence $v^{(n_k)}$, $k\ge1$, say  $v^*=\lim_{k\to\infty}v^{(n_k)}\in  \cap_{n=1}^{\infty}K_n$ (the intersection is closed). 
 To show uniqueness, suppose there is another converging subsequence $v^{(m_k)}$,  $k\ge1$, such that $\widetilde v^*=\lim_{k\to\infty}v^{(m_k)}\in  \cap_{n=1}^{\infty}K_n$ and $v^*\not=\widetilde v^*$. Then, it follows that $\lVert v^* \rVert_2^2= \lVert\widetilde v^* \rVert_2^2$, since the norm is non-decreasing along  the full sequence. By convexity of $K_n$, the intersection $\cap_{n=1}^{\infty}K_n$ is convex and $v_\alpha^*=\alpha v^*+(1-\alpha)\widetilde v^*\in \cap_{n=1}^{\infty}K_n$ for $\alpha\in(0,1)$. By strict convexity of the norm  and $v^*\not=\widetilde v^*$, then
\begin{equation}\label{eq:alphaconvex}
\|v_\alpha^* \|_2^2 < \alpha \|v^* \|_2^2 +(1-\alpha)\|\widetilde v^* \|_2^2=\|v^* \|_2^2,\quad \alpha\in(0,1).
\end{equation}
Let $v_\alpha^{(k)}=\alpha v^{(n_k)}+(1-\alpha) v^{(m_k)}$. By convexity and monotonicity of $K_n$, we have  
$$v_\alpha^{(k)}\in K_{\min\{n_k,m_k\}}\quad\text{for}\quad k\ge 1,$$
$$v_\alpha^{(k)}\to v_\alpha^*\in \cap_{n=1}^\infty K_n\quad\text{for}\quad k\to\infty.$$
By assumption, $\|v_\alpha^{(k)}\|_2^2\ge\min\{\|v^{(n_k)}\|_2^2,\|v^{(m_k)}\|_2^2\}$. Hence, $\|v_\alpha^* \|_2^2\ge\| v^*\|_2^2$,   contradicting \eqref{eq:alphaconvex}.  Consequently, $v^*=\widetilde v^*$.

Since $v^*\in\cap_{n=0}^\infty K_n$, then $v^*=(\pi(\sL),\ldots,\pi(\sU))$ for some non-zero stationary measure $\pi$ of $(\Omega,\cF)$ on $\N_0$.
\end{proof}

If the process is downwardly skip-free, then $\sL=\sU$ and $\pi(\sL)$ might be put to one. Consequently, $\pi(\ell)$, $\ell\ge0$, can  be found recursively from \eqref{Eq-3}. Hence, it only makes  sense to apply the optimization scheme for $\sL<\sU$.

The quadratic minimizing function is  chosen out of convenience to identify a unique element of the set $K_n$. Any strictly convex function could be used for this. If there  exists a unique stationary measure, then one might  choose a random element of $K_n$ as any sequence of elements in $K_n$, $n\in \N$, eventually  converges to the unique element of $ \cap_{n=1}^{\infty}K_n$.  If there are more than one stationary measure, different measures might in principle be found by varying the convex function. In the case, two different stationary measures are found, then any linear combination with positive coefficients is also a stationary measure.

 In practice, the convex constrained optimization approach outlined in Theorem \ref{th:scheme1} often fails for (not so) large $n$,   see the example in Section \ref{sec:examples}. This is primarily because   the coefficients $\gamma_j(\ell)$   become exponentially large with alternating signs, and  because numerical evaluation of close to zero probabilities might return close to zero negative values, hence violating the non-negativity constraint of the convex constrained optimization problem. The numerical difficulties in verifying the inequalities are non-negative are most severe for large $n$, in particular  if $\pi(n)$ vanishes for large $n$.  To face the problems mentioned above, we  investigate  an alternative approach  to the optimization problem.

\begin{lemma}\label{lem:scheme3} 
Assume $\emph{({\bf A1})}$-$\emph{({\bf A3})}$. Define the sets $M_n$, $n\ge \sU$, by
\begin{align*}
M_n&=\left\{v\in\R^{\sU-\sL+1}_{\ge0} \colon \sum_{j=\sL}^{\sU}v_j\gamma_j(\ell)= 0\,\,\text{for}\,\,  \ell=n-(\sU-\sL)+1,\ldots,n,\   \lVert v\rVert_1=1\right\}.\end{align*}
If $M_n\not=\emptyset$, then there is a unique  minimizer  $w^{(n)}$ to the following constraint quadratic optimization problem,
\begin{equation*} 
\min_{v\in M_n} \lVert v\rVert_2^2.
\end{equation*}

 Moreover, if   $\om_+=1$, then $M_n$ is a singleton set,  and $M_n\subseteq K_n$  for $n\ge \sU$,   where $K_n$ is as in \eqref{eq:Kn}. Furthermore,  if there exists a unique  stationary measure $\pi$   of $(\Omega,\mathcal{F})$ on $\N_0$,     then
 $w^*=\lim_{n\to\infty}w^{(n)}\in \cap_{n=1}^{\infty}K_n$ exists. 
In particular,  $w^*$  equals the generator of $\pi$,  that is, $w^*=(\pi(\sL),\ldots, \pi(\sU))$, appropriately normalized.  
\end{lemma}

\begin{proof}
Existence of the minimizer follows similarly to the proof of Theorem \ref{th:scheme1}. 
If $\om_+=1$, then it follows from Proposition \ref{pro-1} below that $M_n$  is a singleton set for $n\ge \sU$.   It follows from Lemma \ref{lem:auxresult} that $M_n\subseteq K_n$.
 Since $w^{(n)}\in  M_n\subseteq K_n $, and $\cap_{n=1}^{\infty}K_n$ contains the generator of the unique stationary measure $\pi$ only, then $v^*=\lim_{n\to\infty} w^{(n)}$ exists and equals the generator of $\pi$.
 \end{proof}

We refer to the optimization problem outlined in  Lemma \ref{lem:scheme3} as the linear approximation scheme.  For $\om_+=1$, a solution $v^{(n)}$ to the linear appproximatioin scheme automatically fulfils $\nu(\ell)=\sum_{j=\sL}^\sU v^{(n)}_j \gamma_j(\ell)\ge 0$ for all $\ell=0,\ldots,n$. In general, these inequalities need to be validated.

\begin{proposition}\label{pro-1}
If $\om_+=1$, then $M_n$, $n\ge\sU$, is a singleton set.
\end{proposition}

\begin{proof}
For $n\ge\sU$, let $G(n)$ be the $(\sU-\sL+1)\times(\sU-\sL+1)$ matrix,
\begin{equation*}
G(n) = \begin{pmatrix}
\gamma_\sL(n-(\sU-\sL-1)) & \gamma_{\sL+1}(n-(\sU-\sL-1)) &\cdots& \gamma_\sU(n-(\sU-\sL-1))\\
\vdots&\vdots&&\vdots\\
\gamma_\sL(n-1) & \gamma_{\sL+1}(n-1) &\cdots& \gamma_\sU(n-1)\\
\gamma_\sL(n) & \gamma_{L+1}(n) & \cdots & \gamma_\sU(n)\\
1 & 1 & \cdots & 1
\end{pmatrix},
\end{equation*}
and let $c_i(n)$ be the cofactor   of $G(n)$ corresponding to the $(\sU-\sL+1)$-th row and $i$-th column, for $i=1,\ldots,\sU-\sL+1$. Then, $\det(G(n)) = \sum_{i=1}^{\sU-\sL+1}c_i(n)$, and there exists a unique solution $v^{(n)}$ to $G(n) v = e_{\sU-\sL+1}$, where $e_{\sU-\sL+1}$ is the $(\sU-\sL+1)$-th unit vector in $\R^{\sU-\sL+1}$, if and only if $\det(G(n))\neq0$. If this is the case, then by Cramer's rule,  
\[v_i^{(n)} = \frac{c_i(n)}{\sum_{j=1}^{\sU-\sL+1}c_j(n)},\quad \text{for}\quad i=1,\ldots, \sU-\sL+1,\]
and  hence $v^{(n)}\in\R^{\sU-\sL+1}_{\ge0}$ if and only if all cofactors have the same sign or are zero. Hence, we aim to show at least one cofactor is non-zero and that all non-zero cofactors have the same sign. In the following, for convenience, we say the elements of a sequence $a_1,\ldots,a_m$ have the same sign if all non-zero elements of the sequence have the same sign.

For $n\ge \sU$, define the $(\sU-\sL+2)\times(\sU-\sL+1)$ matrix
\begin{equation*}
\Gamma(n) = \begin{pmatrix}
\gamma_\sL(n-(\sU-\sL)) & \gamma_{\sL+1}(n-(\sU-\sL)) &\cdots& \gamma_\sU(n-(\sU-\sL))\\
\vdots&\vdots&&\vdots\\
\gamma_\sL(n-1) & \gamma_{\sL+1}(n-1) &\cdots& \gamma_\sU(n-1)\\
\gamma_\sL(n) & \gamma_{L+1}(n) & \cdots & \gamma_\sU(n)\\
1 & 1 & \cdots & 1
\end{pmatrix},
\end{equation*}
and the $(\sU-\sL+1)\times(\sU-\sL+1)$ matrices $\Gamma^\ell(n)$, $\ell=0,\ldots,\sU-\sL $, by removing row $m=\sU-\sL+1-\ell$  of $\Gamma(n)$ (that is, the $(\ell+1)$-th row counting from the bottom). For notational convenience, noting that the columns of $\Gamma(n)$ take a similar form, we write these matrices as
\begin{equation*}
\Gamma^{\ell}(n) = \begin{pmatrix}
\gamma_j(n-(\sU-\sL)) \\
\vdots\\
 \gamma_j(n-(\ell+1)) \\
 \gamma_j(n-(\ell-1)) \\
\vdots\\
\gamma_j(n) \\ 
1
\end{pmatrix},\quad \ell=0,\ldots,\sU-\sL.
\end{equation*}
Note that 
\begin{equation}\label{eq:Grelations}
G(n)=\Gamma^{\sU-\sL}(n),\quad\text{and}\quad \Gamma^0(n)=\Gamma^{\sU-\sL}(n-1).
\end{equation}
 Let $\Gamma^\ell_i(n)$ be the $(\sU-\sL)\times(\sU-\sL)$ matrix obtained  by removing the bottom row and the $i$-th column from  $\Gamma^\ell(n)$. Hence,  the  cofactor $C^\ell_i(n)$ of $\Gamma^\ell(n)$ corresponding to the $(\sU-\sL+1)$-th row and $i$-th column is  
\begin{equation}\label{eq:cofactor}
C^\ell_i(n)=(-1)^{\sU-\sL+1+i}\det(\Gamma^\ell_i(n)),\quad\text{and}\quad C^{\sU-\sL}_i(n)=c_i(n),
\end{equation}
$i=1,\ldots,\sU-\sL+1$. By induction, we will show that the signs of  $C^\ell_i(n)$, $i=1,\ldots,\sU-\sL+1$, are the same, potentially with some cofactors being zero, but at least one being non-zero.

Induction basis. For $n=\sU$, we have
\begin{equation*}
\Gamma(\sU) = \begin{pmatrix}
1 & 0 &\cdots& 0& 0\\
0 & 1 &\cdots& 0& 0\\
\vdots&\vdots&&\vdots&\vdots\\
0 & 0 &\cdots&1& 0\\
0 &0 & \cdots & 0&1 \\
1 & 1 & \cdots & 1&1
\end{pmatrix},
\end{equation*}
and it follows by tedious calculation that
$$C^\ell_{\sU-\sL+1-\ell}(\sU)=(-1)^\ell,\quad C^\ell_i(\sU)=0\quad\text{for}\quad \ell\not=\sU-\sL+1-i,$$
$i=1,\ldots,\sU-\sL+1$ and $\ell=0,\ldots,\sU-\sL$. It follows that the $C^\ell_i$s have the same sign  for $\ell$ fixed and all $i=1,\ldots,\sU-\sL+1$ (all cofactors are zero, except for $i=\sU-\sL+1-\ell$). 

We postulate that the non-zero elements fulfil
$$ \text{sign}(C^\ell_i(n))=(-1)^{(n-\sU)(\sU-\sL)+\ell},\quad \text{for}\quad n\ge\sU,$$
 and $\ell=0,\ldots,\sU-\sL$, $i=1,\ldots,\sU-\sL+1$. The hypothesis holds for $n=\sU$.

Induction step. Assume the statement is correct for some $n\ge\sU$. 
Using $m_*=\om_+-\om_--1=\sU-\sL+1$ and \eqref{Eq-1}, we obtain for $\ell=1 ,\ldots,\sU-\sL$ (excluding $\ell=0$),
\begin{equation*}
\Gamma^{\ell}(n+1) = \begin{pmatrix}
\gamma_j(n+1-(\sU-\sL)) \\
\vdots\\
 \gamma_j(n+1-(\ell+1)) \\
 \gamma_j(n+1-(\ell-1)) \\
\vdots\\
\gamma_j(n)\\
\gamma_j(n+1) \\
1 
\end{pmatrix}=\begin{pmatrix}
\gamma_j(n+1-(\sU-\sL)) \\
\vdots\\
 \gamma_j(n+1-(\ell+1)) \\
 \gamma_j(n+1-(\ell-1)) \\
\vdots\\
\gamma_j(n)\\
\sum_{k=1}^{\sU-\sL+1} \gamma_j(n+1-k)f_k(n+1) \\
1 
\end{pmatrix}.
\end{equation*}
Hence,  using the linearity of the determinant, we obtain for $0<\ell\le \sU-\sL$ (excluding $\ell=0$),
\begin{align*}
 \det(\Gamma_i^\ell(n+1)) & =  \begin{vmatrix*}[c]
\gamma_j(n+1-(\sU-\sL)) \\
\vdots\\
 \gamma_j(n+1-(\ell+1)) \\
 \gamma_j(n+1-(\ell-1)) \\
\vdots\\ 
\gamma_j(n)\\
 \gamma_j(n+1-\ell)f_\ell(n+1) 
\end{vmatrix*}+\begin{vmatrix*}[c]
\gamma_j(n+1-(\sU-\sL)) \\
\vdots\\
 \gamma_j(n+1-(\ell+1)) \\
 \gamma_j(n+1-(\ell-1)) \\
\vdots\\
\gamma_j(n)\\
 \gamma_j(n+1-(\sU-\sL+1))f_{\sU-\sL+1}(n+1)
\end{vmatrix*} \\
&=f_\ell(n+1)\begin{vmatrix*}[c]
\gamma_j(n+1-(\sU-\sL)) \\
\vdots\\
 \gamma_j(n+1-(\ell+1)) \\
 \gamma_j(n+1-(\ell-1)) \\
\vdots\\ 
\gamma_j(n)\\
 \gamma_j(n+1-\ell)
\end{vmatrix*}+f_{\sU-\sL+1}(n+1)\begin{vmatrix*}[c]
\gamma_j(n+1-(\sU-\sL)) \\
\vdots\\
 \gamma_j(n+1-(\ell+1)) \\
 \gamma_j(n+1-(\ell-1)) \\
\vdots\\
\gamma_j(n)\\
 \gamma_j(n+1-(\sU-\sL+1))
\end{vmatrix*} \\
&=f_\ell(n+1)(-1)^{\ell-1}\begin{vmatrix*}[c]
\gamma_j(n-(\sU-\sL-1)) \\
\vdots \\
 \gamma_j(n )
\end{vmatrix*} \\
&\qquad\qquad+f_{\sU-\sL+1}(n+1)(-1)^{\sU-\sL-1}\begin{vmatrix*}[c]
\gamma_j(n-(\sU-\sL)) \\
\vdots\\
 \gamma_j(n - \ell ) \\
 \gamma_j(n-(\ell-2)) \\
\vdots\\
 \gamma_j(n )
\end{vmatrix*} \\
&=f_\ell(n+1)(-1)^{\ell-1}\det(\Gamma_i^{\sU-\sL}(n))+f_{\sU-\sL+1}(n+1)(-1)^{\sU-\sL-1}\det(\Gamma_i^{ \ell-1 }(n))
\end{align*}
with the remaining terms from the linear expansion of the determinant are zero.  In the computation of the determinant above, we abuse $\gamma_j(n+1-k)$ for the row vector with  the $i$-th coordinate deleted.  For $\ell=0$, then using \eqref{eq:Grelations},
$$\det(\Gamma^0_i(n+1))=\det(\Gamma^{\sU-\sL}_i(n)).$$
The above conclusions result in the following for the sign of the cofactors, using \eqref{eq:cofactor}:
\begin{align}
 C_i^\ell(n+1)) &=f_\ell(n+1)(-1)^{\ell-1}C_i^{\sU-\sL}(n) \label{eq:cofactor2} \\
& \quad\quad +f_{\sU-\sL+1}(n+1)(-1)^{\sU-\sL-1} C_i^{ \ell-1 }(n),\quad 0<\ell\le\sU-\sL, \nonumber \\
C^0_i(n+1)&=C^{\sU-\sL}_i(n).\label{eq:cofactor3}
\end{align}

We recall some properties of $f_\ell(n)$.  According to Lemma \ref{lem:ck}(vi), $f_{\sU-\sL+1}(n)>0$ for $n\ge\sU+1$, using $\si_+=0$  (otherwise  zero is a trapping state) and $-\om_-=\sU-\sL+1$. For  $0\le \ell<-\om_-$, we have   
 $\text{sgn}(\om_-+\ell+1/2)=-1$, and hence   $f_\ell(n)\le0$ for $n\ge\sU+1$, according to Lemma \ref{lem:ck}(vii)-(viii).
Consequently, the sign of the two terms in   \eqref{eq:cofactor2} are: 
$$(-1)(-1)^{\ell-1}(-1)^{(\sU-\sL)(n-\sU)+(\sU-\sL)}=(-1)^{(\sU-\sL)(n+1-\sU)+\ell},$$
$$(+1)(-1)^{\sU-\sL+1}(-1)^{(\sU-\sL)(n-\sU)+\ell-1}=(-1)^{(\sU-\sL)(n+1-\sU)+\ell},$$
hence the sign of  $C_i^\ell(n+1)$ corroborates the induction hypothesis. The  sign  of the term in \eqref{eq:cofactor3} is
$$(-1)^{\sU-\sL}(-1)^{(\sU-\sL)(n-\sU)}=(-1)^{(\sU-\sL)(n+1-\sU)+0}, $$
again in agreement with the induction hypothesis. 

It remains to show that at least one cofactor is non-zero, that is,  $C_i^{\sU-\sL}(n)\neq0$  for   at least one $1\le i\le \sU-\sL+1$ and $n\ge \sU$. Let $a_{n,\ell}=\sum_{i=1}^{\sU-\sL+1}|C_i^{\ell}(n)|$.   From \eqref{eq:Grelations} and \eqref{eq:cofactor2}, we have
\begin{align}\label{Eq:recurrence-a_n}
a_{n+1,\ell}&=|f_{\ell}(n+1)|a_{n,\sU-\sL}+|f_{\sU-\sL+1}(n+1)|a_{n,\ell-1},\quad 1\le\ell\le \sU-\sL, \\
 a_{n+1,0}&=a_{n,\sU-\sL}, \nonumber
\end{align}
for $n\ge \sU.$  We show by induction that $a_{n,\ell}\not=0$ for $n\ge \sU$ and $0\le\ell\le \sU-\sL$. Hence, the desired conclusion follows. For $n=\sU$, we have $a_{\sU,\ell}=1$ for all $\ell=0,\ldots,\sU-\sL$. Assume $a_{n,\ell}\not=0$ for   $\ell=0,\ldots,\sU-\sL$ and some $n\ge\sU$. Since $f_{\sU-\sL+1}(n+1)>0$ for $n+1\ge \sU+1$, then it follows from \eqref{Eq:recurrence-a_n} that $a_{n+1,\ell}\not=0$ for   $\ell=0,\ldots,\sU-\sL$. The proof is completed. 
\end{proof}

\section{The generating function}

 In this and the next   section, we consider one-species SRNs mainly with mass-action kinetics, without assuming (${\bf A3}$). That is, we allow for multiple PICs indexed by $s = \max\{\si_+,\so_-\},\ldots,\so_-+\om_*-1$ and general $\om_*\ge 1$. For a stationary measure $\pi$ of $(\Omega,\cF)$ on a PIC $\om_*\N_0+s$ (for some $s$), we define 
\begin{equation}\label{eq:pis}
\pi_s(\ell)=\pi(s+\ell\om_*),\quad \text{for}\quad \ell\in\N_0.
\end{equation}
 Then, $\pi_s$ is  a stationary measure of $(\Omega_*,\cF_s)$ on $\N_0$ for the translated Markov chain as in Section \ref{sec:charac}. Similarly, we add the superscript $s$ to $\gamma_j(\ell), c_k(\ell), f_k(\ell)$, whenever convenient. These quantities are already expressed in terms of the translated Markov chain.

For one-species SRNs with  mass-action kinetics, one might represent the stationary distribution (when it exists) through the probability generating function (PGF). Recall that for a random variable $X$ on $\N_0$ with law $\mu$, the PGF is:
$$g(z)=\mathbb{E}(z^X)=\sum_{x\in\N_0}\mu(x) z^x,\quad |z|\le 1,$$
such that
$$\mu(x)=\frac{g^{(x)}(0)}{x!}, \quad x\in\N_0,$$
where $g^{(x)}$ is the $x$-th derivative of $g$.

\begin{proposition}\label{pro-4}
Let  a one-species SRN fulfilling $\emph{({\bf A1})}$ with  mass-action kinetics be given, hence $\emph{({\bf A2})}$ also holds. Assume there exists a  stationary distribution $\pi$ on  $\om_*\N_0+s$, for some $s\in\{\max\{\si,\so\},\ldots,\so+\om_*-1\}$. Let $\pi_s(\ell)=\pi( \om_*\ell+s)$, $\ell\in\N_0$, be the corresponding distribution on $\N_0$ and $g_s$ the  PGF of $\pi_s$. Then, $g_s$ solves the following differential equation:
$$\sum_{y_k\to y_k' \in \mathcal{R}} \kappa_k(z^{y_k'}-z^{y_k})g_s^{(y_k)}(z)=0, \quad |z|\le 1.$$
\end{proposition}

\begin{proof}
Recall that $\pi_s$ is an equilibrium of the master equation:
\begin{align*}
\sum_{y_k\to y_k'\in \mathcal{R}}\lambda_k(s+x\om_*-\xi_k)\pi_s(x-\xi_k\om_*^{-1})-\sum_{y_k\to y_k'\in \mathcal{R}}\lambda_k(s+x\om_*)\pi_s(x)=0.
\end{align*}
Multiplying through by $z^{s+x\om_*}$ and summing over $x\in\N_0$, we obtain
\begin{align*}
\sum_{y_k\to y_k'\in \mathcal{R}}\sum_{x=0}^{\infty}\lambda_k(s+x\om_*-\xi_k)\pi_s(x-\xi_k\om_*^{-1})z^{s+x\om_*}-\sum_{y_k\to y_k'\in \mathcal{R}}\sum_{x=0}^{\infty}\lambda_k(s+x\om_*)\pi_s(x)z^{s+x\om_*}=0.
\end{align*}
Using the form of mass-action kinetics, and rearranging terms,
\begin{align*}
\sum_{y_k\to y_k'\in \mathcal{R}}\kappa_kz^{y_k'}\sum_{x=0}^{\infty} (s+x\om_*)^{\underline{y_k}}\pi_s(x)z^{s+x\om_*-y_k}-\sum_{y_k\to y_k'\in \mathcal{R}}\kappa_k z^{y_k}\sum_{x=0}^{\infty} (s+x\om_*)^{\underline{y_k}}\pi_s(x)z^{s+x\om_*-y_k}=0,
\end{align*}
where $x^{\underline{y}}=x(x-1)\ldots(x-y+1)$. Here, we have used that $\pi_s(n)=0$ for $n<0$, and if $\xi_k<0$, then $s-\xi_k-\om_*\le\so-1-\xi_k<y_k'-\xi_k=y_k$, which implies that $(s+n\om_*)^{\underline{y_k}}=0$ for $0\le n<-\xi_k\om_*^{-1}$. As $g_s(z)=\sum_{x=0}^{\infty} \pi_s(x) z^{s+x\om_*}$, we may write this as
\begin{align*}
\sum_{y_k\to y_k'\in \mathcal{R}}\kappa_kz^{y_k'}g_s^{(y_k)}(z)-\sum_{y_k\to y_k'\in \mathcal{R}}\kappa_k z^{y_k}g_s^{(y_k)}(z)=0,
\end{align*}
whence upon collecting terms yields the desired.
\end{proof}

Since
$$\pi_s(j)=\frac{g_s^{(s+j\om_*)}(0)}{(s+j\om_*)!},\quad\text{for all}\quad j\in\N_0,$$
 we arrive at the following alternative expression for the stationary distribution.

\begin{corollary}\label{cor-pgf}
Assume as in Proposition \ref{pro-4}. Then
$$\pi_s(\ell)=\sum_{j=\sL_s}^{\sU_s}\frac{g_s^{(s+j\om_*)}(0)}{(s+j\om_*)!}\gamma^s_j(\ell), \quad \ell\in\N_0,  $$
where $\gamma^s_j$ is defined as in \eqref{Eq-1}, and $\pi_s$ as in \eqref{eq:pis}.
\end{corollary}

\begin{example}\label{ex:53}
Consider the SRN with mass-action kinetics,
\[0\ce{->[\ka_1]} \tS \ce{->[\ka_2]} 2\tS\ce{->[\ka_3]}0,\]
with $\ka_1,\ka_2,\ka_3>0$.
This is an upwardly skip-free process: $\om_*=1$, $\om_+=1$, $\om_-=-2$, $s=0$, $\sL_0=0$, $\sU_0=1$.

Applying Proposition~\ref{pro-4}, the governing differential equation of the PGF is,
\[
\ka_3(1-z^2)g''_0(z)+\ka_2(z^2-z)g_0'(z)+\ka_1(z-1)g_0(z)=0,\quad |z|\le 1,\]
that is,
\begin{equation}\label{Eq-7}
(1+z)g''_0(z)-\frac{\ka_2}{\ka_3}zg_0'(z)-\frac{\ka_1}{\ka_3}g_0(z)=0.\end{equation}
Let  $\tilde{z}=\frac{\ka_2}{\ka_3}(z+1)$ and $h(\tilde{z})=g_0(z)$, then  \eqref{Eq-7} yields the Kummer differential equation \cite{AS72}:
\[\tilde{z}h''(\tilde{z})+\lt(\frac{\ka_2}{\ka_3}-\tilde{z}\rt)h'(\tilde{z})-\frac{\ka_1}{\ka_2}h(\tilde{z})=0.\]
Taking $g_0(1)=1$,
\[g_0(z)=\frac{{}_1F_1\!\left(\frac{\ka_1}{\ka_2};\frac{\ka_2}{\ka_3};\frac{(z+1)\ka_2}{\ka_3}\right)}{{}_1F_1\!\left(\frac{\ka_1}{\ka_2};\frac{\ka_2}{\ka_3};\frac{2\ka_2}{\ka_3}\right)},\]
where
$${}_1F_1(a;b;z)=\sum_{n=0}^{\infty}\frac{a^{\overline{n}}z^n}{b^{\overline{n}}n!},$$
 is the {\em confluent hypergeometric function} (also called {\em Kummer's function}) of the first kind \cite{AS72}, and $a^{\overline{n}}=\frac{\Gamma(a+n)}{\Gamma(a)}$ is the ascending factorial.
Using the relation between Kummer's function and its derivatives
\[\frac{d {}_1 F_1^{(n)}(a;b;z)}{ d z^n}=\frac{a^{\overline{n}}}{b^{\overline{n}}}{}_1 F_1(a+n;b+n;z),\quad n\in\N,\]
it follows that
\begin{equation*}
\pi(x)=\frac{g_0^{(x)}(0)}{x!}=\frac{1}{x!}\lt(\frac{\ka_2}{\ka_3}\rt)^{\!x}\frac{\Gamma(\frac{\ka_2}{\ka_3})}{\Gamma(\frac{\ka_1}{\ka_2})}\frac{\Gamma(x+\frac{\ka_1}{\ka_2})}
{\Gamma(x+\frac{\ka_2}{\ka_3})}\frac{{}_1F_1(x+\frac{\ka_1}{\ka_2};x+\frac{\ka_2}{\ka_3};\frac{\ka_2}{\ka_3})}{{}_1F_1(\frac{\ka_1}{\ka_2};\frac{\ka_2}{\ka_3};\frac{2\ka_2}{\ka_3})}.
\end{equation*}
In particular, $\ka_1\ka_3=\ka_2^2$ if and only if  $\pi$ is a Poisson distribution,
$$\pi(x)=\frac{1}{x!}\left(\frac{\ka_1}{\ka_2}\right)^{\!x}e^{-\frac{\ka_1}{\ka_2}}.$$
When this is the case, the reaction network is \emph{complex balanced} and the form of the stationary distribution is well known \cite{ACK10}.

On the other hand,
$$ g_0(0)=\frac{{}_1F_1\!\left(\frac{\ka_1}{\ka_2};\frac{\ka_2}{\ka_3};\frac{\ka_2}{\ka_3}\right) }{{}_1F_1\!\left(\frac{\ka_1}{\ka_2};\frac{\ka_2}{\ka_3};\frac{2\ka_2}{\ka_3}\right)}, \quad g'_0(0)=\frac{\ka_1}{\ka_2}\frac{{}_1F_1\!\left(1+\frac{\ka_1}{\ka_2};1+\frac{\ka_2}{\ka_3};\frac{\ka_2}{\ka_3}\right) }{{}_1F_1\!\left(\frac{\ka_1}{\ka_2};\frac{\ka_2}{\ka_3};\frac{2\ka_2}{\ka_3}\right)},$$
and by Corollary~\ref{cor-pgf},
\begin{equation*}
\pi(x)=g_0(0)\gamma^0_0(x)+g_0'(0)\gamma^0_1(x),\quad x\in\N_0.
\end{equation*}
\end{example}

\section{Examples}\label{sec:examples}

To end we present some examples using the linear approximation scheme and the convex constrained optimization approach.
We use the criteria in  \cite[Theorem 7]{HWX}  to check whether a CTMC with polynomial transition rates is positive  recurrent, null recurrent, transient and non-explosive, or explosive. These properties hold for either all PICs or none, provided the transition rate functions are polynomials for large state  values, as in mass-action kinetics  \cite[Theorem 7]{HWX}. 

We have made two implementations of the numerical methods. One in \texttt{R} and one in    \texttt{Python} (only mass-action kinetics) with codes available on request. In all examples, the  code runs in few seconds   on a standard laptop. The two codes agree when run on the same example. We have not aimed to optimize for speed.  In figures, `State $x$' refers to the state of the original Markov chain, and `Index $n$'  refers to the translated Markov chain, the index of, for example,  $\gamma_j^s(n)$. The index $s$ refers to the irreducibility class in Proposition \ref{pro-1}.
  
\begin{example}
Consider the SRN with mass-action kinetics,
\begin{equation*}
 \tS\ce{->[\ka_1]} 2\tS \ce{<=>[\ka_2][\ka_4]} 3\tS\ce{->[\ka_3]}\tS.
\end{equation*}
We have $\omega_+=1$ and $\omega_-=-2$   with $s=1$ (zero is a neutral state), and $\sL_1=0, \sU_1=1$. Furthermore, a unique stationary distribution  exists  since the SRN is positive recurrent for all positive rate constants \cite[Theorem 7]{HWX}, and $\pi_1(x)=\pi(1+x)$. As the  reaction network is weakly reversible (the reaction graph consists of a finite union of disjoint strongly connected components), then it is complex balanced for specific choices of rate constants, yielding a Poisson stationary distribution \cite{ACK10}. This is the case if and only if $\ka_1(\ka_3+\ka_4)^2=\ka_2^2\ka_3$.

Here, we focus on the stability of the numerical approximations using  the  linear approximation scheme and convex constrained optimization for  a single set of parameters, $\ka_1=40, \ka_2=22, \ka_3=\ka_4=1$, see Figure \ref{fig:fig1}.
Convex constrained optimization fails for $n>18$ in \eqref{Eq-4} due to exponentially increasing  $\gamma_j^1(\ell)$ values with alternating signs. In contrast, the linear approximation scheme is quite robust and returns accurate estimates for the generating terms $\pi_1(0),\pi_1(1)$ ($=\pi(1),\pi(2))$, even for $n=70$. However, in this situation,  inaccurate and negative probability values for large state values are returned, see Figure \ref{fig:fig1}. The estimated values of $\pi(32)$ and $\pi(33)$ are zero to the precision of the computer and the first negative estimate is $\pi(34)=-8.4\cdot 10^{-13}$. From then on, the estimated probabilities increase in absolute value.  The estimated  generating terms for the convex constrained optimization problem with $n=18$ and the linear approximation scheme with $n=25$ deviate on the seventh decimal point only. In the latter case, the estimates remain unchanged for $25\le n\le 70$  for up to seven decimal points, despite negative probabilities are found for large  $n$.

 It is remarkable that for $n=70$ with $\gamma_j^1(n)$ of the order $e^{50}\approx 10^{22}$, we still numerically find that $M_n$ is a singleton set, as postulated in Proposition~\ref{pro-1}, despite the solution gives raise to instabilities in calculating the probabilities. Also the numerical computations confirm that the limit in \eqref{eq:lim} in Lemma \ref{le:stationary-measure} is zero, as  $\gamma_j^1(n)$ increases beyond bound.
\end{example}

\begin{figure}[h!]
\begin{center}
\includegraphics[width=5in]{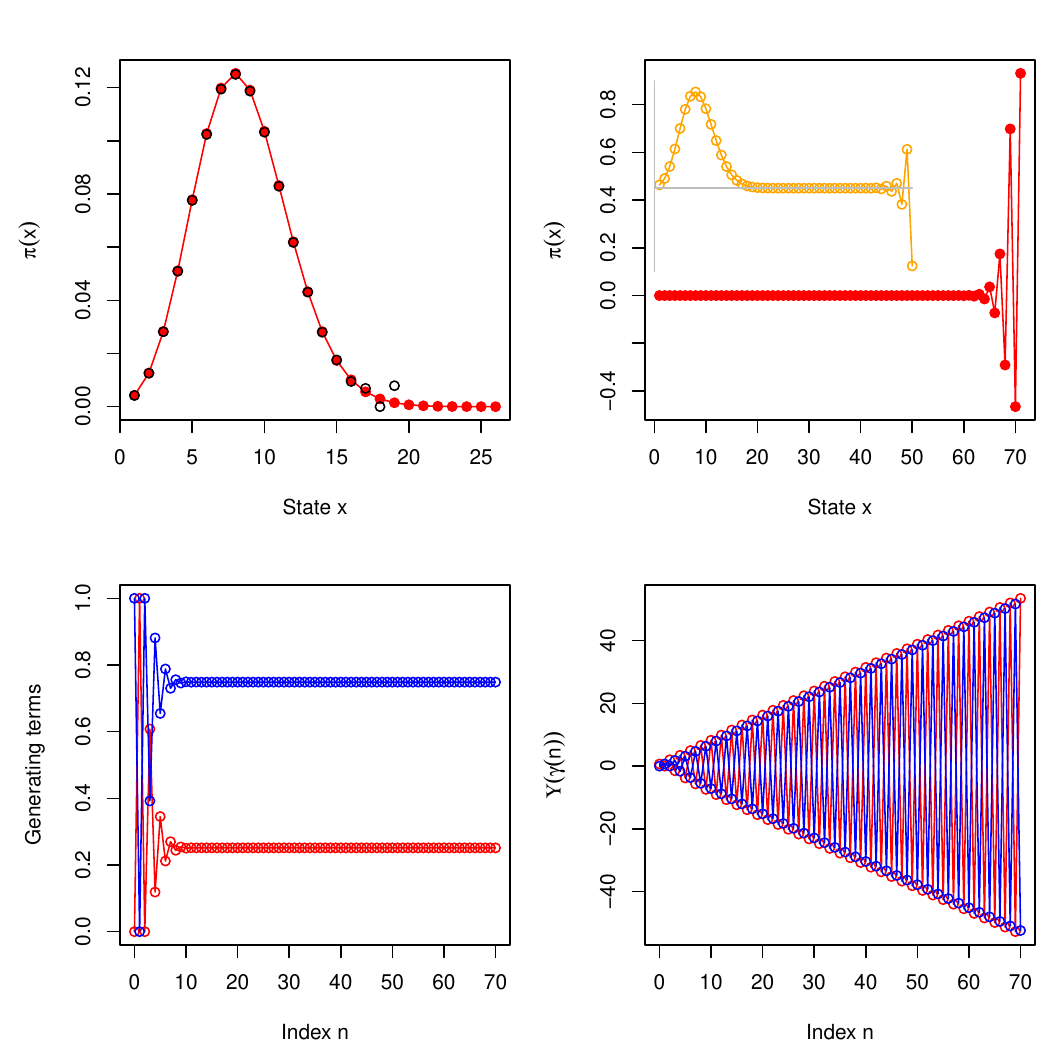}
\caption{Top left:  The stationary distribution calculated using the linear approximation scheme with  $n=25$ (red dots) and convex constrained optimization with $n=18$ (black circles). The latter results in wrong probability estimates for the states $x=18$ and $x=19$. Top right: The stationary distribution calculated using the linear approximation scheme with  $n=70$ (red dots). The orange subfigure is a blow up of the first 50 states for $n=70$, normalized to their sum, indicating the correct form of the distribution is retrieved even in the presence of instability, and the onset of instabilities. Bottom left:  Approximate values of $\pi_1(\sL_1)$ and $\pi_1(\sU_1)$ as function of $n$ found using the linear approximation scheme. Convergence is very fast. Bottom right: The values of $\gamma^1_{\sL_1}(n)$ and  $\gamma^1_{\sU_1}(n)$ for $n =0,\ldots,70$. The coefficients are plotted as $\Upsilon(x)=\log(x+1)$ for $x\ge 0$ and $\Upsilon(x)=-\log(-x+1)$ for $x\le 0$. Dots are connected by lines for convenience. }
\label{fig:fig1}
\end{center}
\end{figure}

The following example demonstrates that both the linear approximation scheme and the convex optimization approach can be efficient in practice for ergodic CTMCs.

\begin{example}
For the SRN with mass-action kinetics,
\begin{equation}\label{cycle}
    \xymatrix{  \tS \ce{->[\ka_1]}  3\tS\ce{ ->[\ka_2]}  2\tS \ce{->[\ka_3]}   4\tS \ce{->[\ka_4]}  \tS}
\end{equation}
we obtain $\om_+=2$, $\om_-=-3$, and $s=1$, $\sL_1=0$, $\sU_1=2$, such that there is one PIC with state space $\N$.
 Despite Proposition \ref{pro-1} does not apply (as $\om_+\not=1$), numerically we find that $M_n$ is a singleton set.  Using the linear approximation scheme or the convex optimization approach, we obtain a rather quick convergence, see Figure~\ref{fig:fig2}. In this case, the coefficients $\gamma^1_j(\ell)$, $j=0,1,2$, decrease fast towards zero, taking both signs. Both algorithms run efficiently even for large $n$ as the coefficients vanish.  The bottom right plot in Figure \ref{fig:fig2} shows $\gamma(n)/\|\gamma(n)\|_1$. There appears to be a periodicity of $\sU_1-\sL_1+1=3$, demonstrating numerically that the matrices  $A(3n)$, $A(3n+1)$ and $A(3n+2)$, $n\in\N_0$, each converges as $n\to\infty$, see Theorem \ref{th:upward}.  The generator recovered from either of the three sequences $A(3n)$, $A(3n+1)$ and $A(3n+2)$ agree to high accuracy, and agree with the generator found using the linear approximation scheme.
\end{example}

\begin{figure}[h!]
\begin{center}
\includegraphics[width=4in]{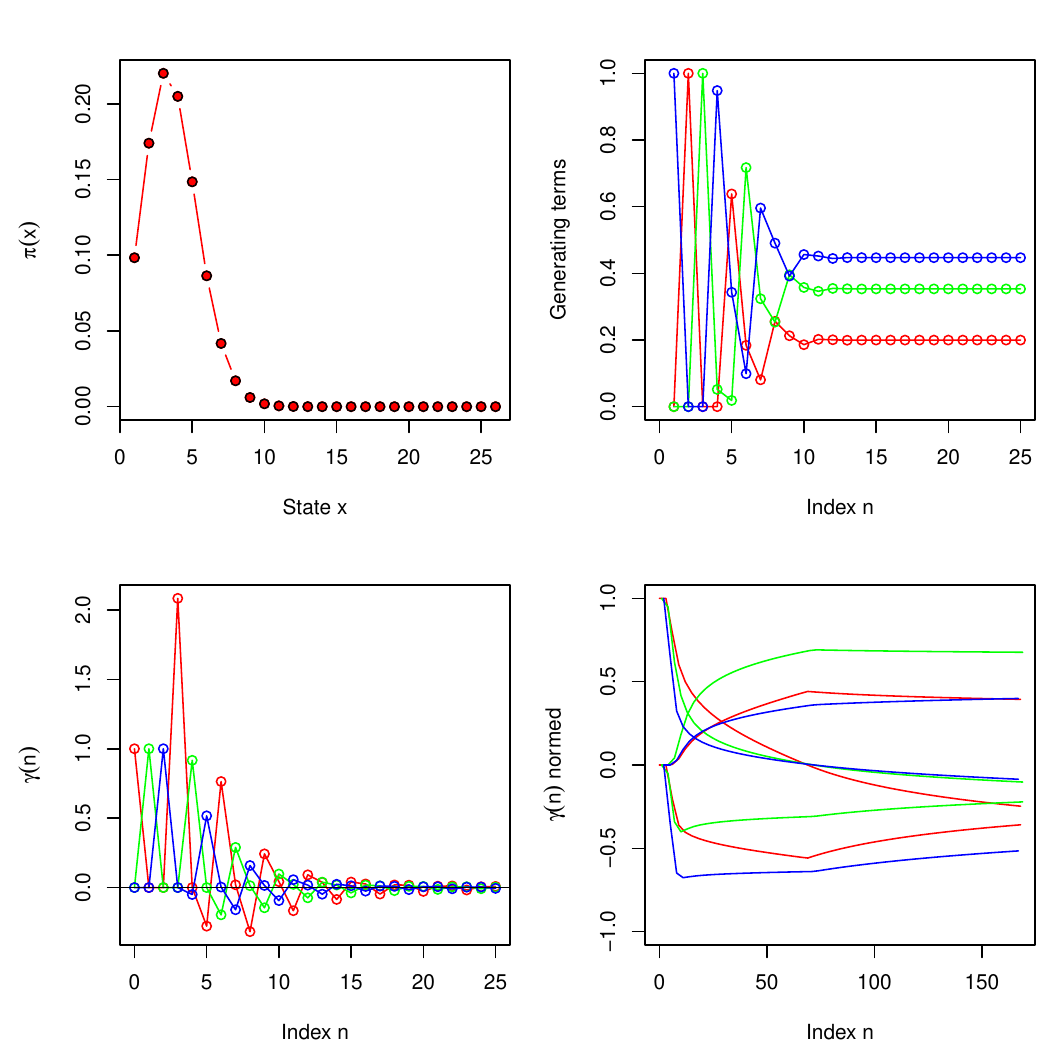}
\caption{Top left: Stationary distribution for \eqref{cycle} with $\ka_1= 50$, $\ka_3=5$ and $\ka_2=\ka_4=1$. The linear approximation scheme is shown in red, while the convex constrained optimization scheme is overlaid in black. Dots are connected by lines for convenience. Top right and Bottom left: Convergence of the generating terms is fast as $\gamma_j^s(\ell)$ decreases fast to zero with $\ell$ becoming large. Bottom right: Shown is $\gamma(3n)/\|\gamma(3n)\|_1$ (red), $\gamma(3n+1)/\|\gamma(3n+1)\|_1$ (green), and $\gamma(3n+2)/\|\gamma(3n+2)\|_1$ (blue).  The numerical computations clearly indicate periodicity. Note that despite convergence has not been achieved for the simulated values of $n$, the generators recovered from the three series $A(3n),A(3n+1), A(3n+2)$  agree with those found from the linear approximation scheme, see Theorem \ref{th:upward}. Dots are replaced by lines for visual reasons. }
\label{fig:fig2}
\end{center}
\end{figure}

Although, in theory, it seems to make  sense to use the linear approximation scheme \emph{only} for stationary measures for which $\pi(n)/\|\gamma(n)\|$ is vanishing. In practice, it seems that the linear approximation scheme still captures the feature of a stationary measure with non-vanishing $\pi(n)/\|\gamma(n)\|_1$ decently, when the states are not too large.

\begin{example} 
Computing an unknown distribution. We give an example of  a mass-action SRN,
$$0\ce{->[10]}\tS\ce{->[12]}2\tS\ce{->[1]}6\tS,\quad 2\tS\ce{->[2]}0,$$
which is null recurrent by the criterion in  \cite[Theorem 7]{HWX}, and hence there exists a unique stationary measure due to \cite[Theorem~3.5.1]{markov} and \cite[Theorem~1]{D55}. In this case, $\omega_1=-2$ and $\omega_+=4$. Furthermore, $s=0$ and $\sL_0=0$, $\sU_0=1$. We apply the linear approximation scheme to the SRN with $n=150$ and find that $M_n$ is a singleton set, despite $\om_+\not=1$, see Proposition \ref{pro-1}. For large states the point measures  are considerably far from zero, see  Figure \ref{fig:fig4}. Moreover, instabilities occur.  The inaccuracies in the values are due to small inaccuracies in the estimation of the generating terms and the large coefficients $\gamma_j(\ell)$ that increases exponentially.  
\end{example}

\begin{figure}[h!]
\begin{center}
\includegraphics[width=4in]{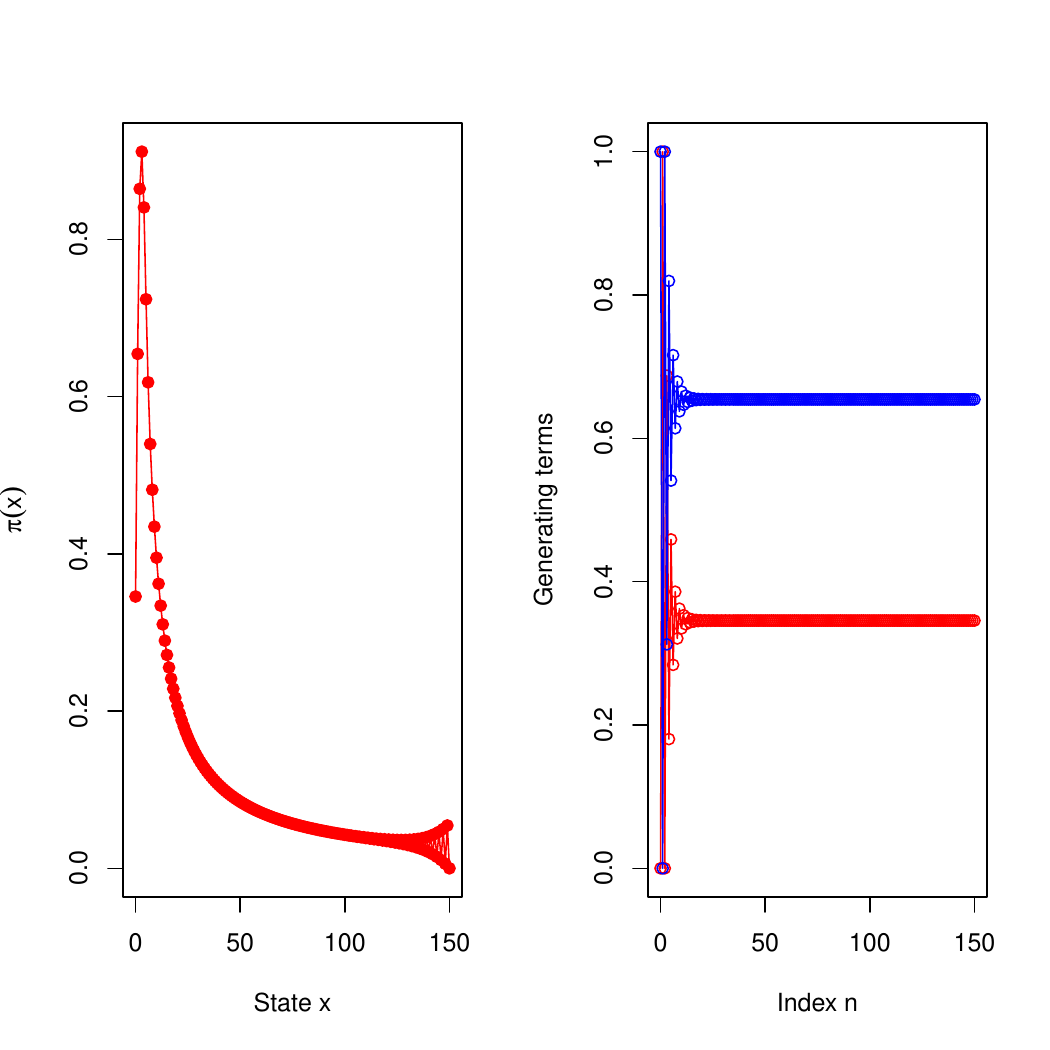}
\caption{Left: The stationary measure computed with the linear approximation scheme and $n=150$. For large states, significant errors occur in the estimation. Right: The generating terms.  }  
\label{fig:fig4}
\end{center}
\end{figure}
 
 We know from Corollary~\ref{Cor:uniqueness} that there exists a unique stationary measure for CTMCs with polynomial transition rates if $\om_-=-2$ and $\om_+=1$, it remains to see if such a stationary measure is finite.  With the aid of our numerical scheme, we might be able to infer this information in some scenarios.

 \begin{example} 
Computing  an unknown measure. Consider the following SRN,
$$10\tS\ce{->[5000]}12\tS\ce{->[10]}13\tS\ce{->[1]}16\tS,\quad 13\tS\ce{->[1]}10\tS.$$
 It is explosive  by the criterion in \cite[Theorem 7]{HWX}.
We have $s=10$, $\omega_-=-3$, $\omega_+=3$, $\sL_{10}=0$ and $\sU_{10}=2$.  The linear approximation scheme retrieves what appears to be a stationary distribution, see Figure \ref{fig:fig5}. The numerical computations confirm that the limit in \eqref{eq:lim2} in Theorem \ref{th:upward} is zero, pointing to the stationary distribution being unique.
\end{example}

\begin{figure}[h!]
\begin{center}
\includegraphics[width=4in]{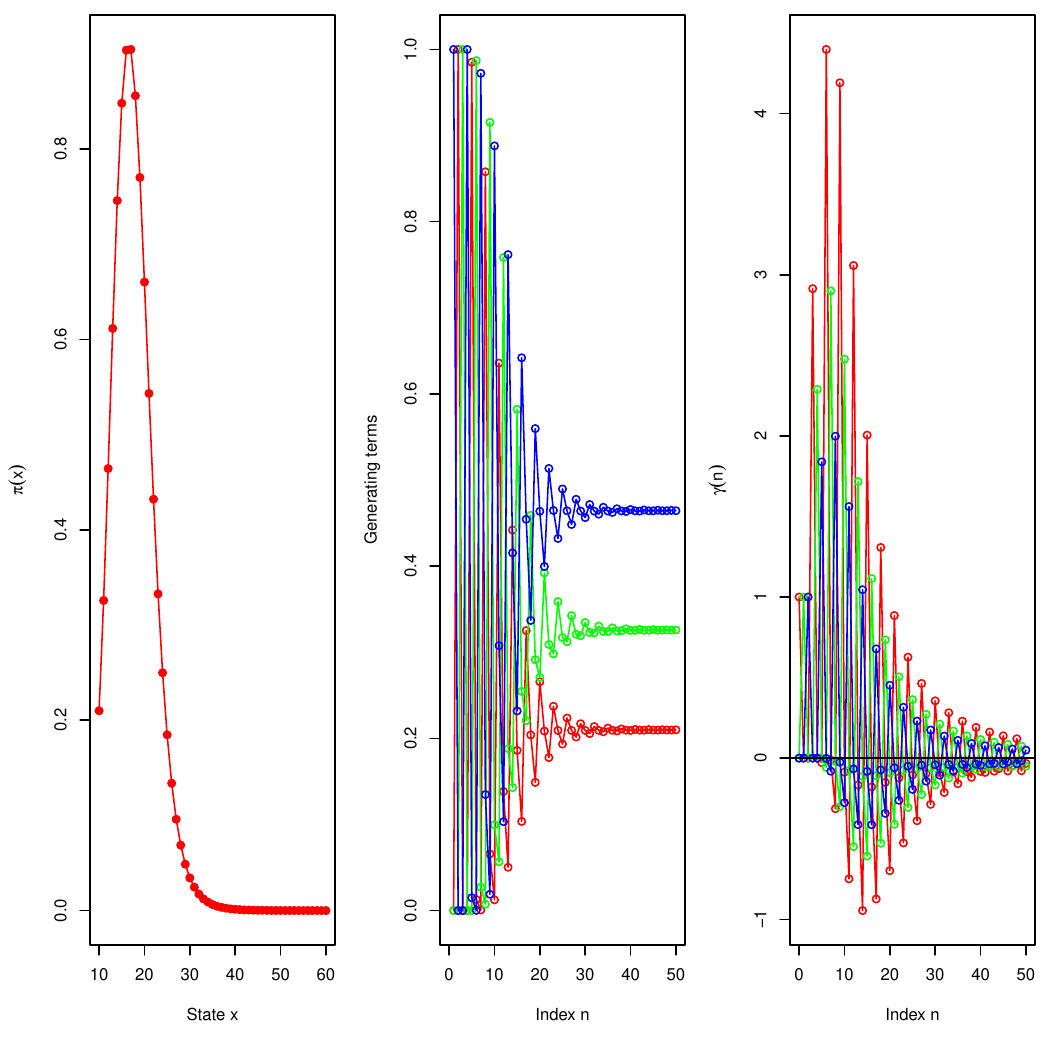}
\caption{The stationary measure of an explosive SRN, the generating terms and the coefficients $\gamma_j^s(\ell)$.  }
\label{fig:fig5}
\end{center}
\end{figure}
 
    We end with an example for which there exists more than one stationary measure.

\begin{example}\label{ex:nonunique}
Consider the SRN with reactions $0\ce{->[\lambda_1]} \tS$, $2\tS\ce{->[\lambda_{-2}]} 0$, and 
\begin{align*}
\lambda_1(0)&=\lambda_1(1)=1,\quad \,\,\,\,\lambda_1(x)=\lfloor\tfrac{x}{2}\rfloor^2,\quad x\ge 2,\\
\lambda_{-2}(0)&=\lambda_{-2}(1)=0,\quad \lambda_{-2}(x)=\lfloor\tfrac{x+1}{2}\rfloor^{-2},\quad x\ge 2,
\end{align*} 
 Then, ({\bf A1})-({\bf A3}) are fulfilled  with $\om_-=-2$, $\om_+=1$, $s=0$, $\sL_0=0$, $\sU_0=1$ (so $\pi_0=\pi$). From Corollary \ref{cor:fraction} with $x=\sU_0+2=3$,
$$Q_n(3)=\prod_{i=0}^n\frac{\lambda_1(2+2i)\lambda_{-2}(3+2i)}{\lambda_1(3+2i)\lambda_{-2}(4+2i)}=1,$$
$$H(3)=\sum_{n=0}^\infty \frac{1}{(n+1)^2}+ \frac1{(n+2)^2}=-1+\sum_{n=1}^\infty \frac{2}{n^2}=\frac{\pi^2}3-1<\infty.$$
Consequently, there is not a unique stationary measure of $(\Omega,\cF)$ on $\N_0$.   Numerical computations suggest that $[\Psi_2(3),\Psi_1(3)]\approx [1.5351, 2.6791],$ 
using \eqref{Eq-recurrence} with $k=700$. See Corollary \ref{cor:fraction} for a definition of $\Psi_1,\Psi_2$.

In the following, we will explore the stationary measures using the theory of Section \ref{sec:skipup} and compare it to results obtained from the linear approximation scheme. For any stationary measure $\pi$, by  definition of $\phi(x)$  and   $\phi(x)<h(x)$, we have for $x\ge 3$,
\begin{align*}
\frac{\pi(x)}{\pi(x-1)}=&\frac{\lambda_{-2}(x-1)}{\lambda_{-2}(x)}\frac{1}{\phi(x)}
>\frac{\lambda_{-2}(x-1)}{\lambda_{-2}(x)}\frac{1}{h(x)}=\left(\Big\lfloor\frac{x-1}{2}\Big\rfloor\Big\lfloor\frac{x+1}{2}\Big\rfloor\right)^{\!2},
\end{align*} 
hence
$$\pi(x)> C \frac{(x!)^4}{16^x},\quad x\ge 0,$$
for some constant $C>0$, and $\pi$ is not a distribution. This is illustrated in Figure \ref{fig:fig6} (top left) that shows the logarithm of $\pi(x)$ with $\phi^*=2.67$, using Corollary \ref{cor:fraction}.  The red line is a fitted curve to $\log(\pi(x))$ for the  \emph{even} states: $\log(\pi(2x))\approx 3.504+2.009 x\log(x)-3.429x$, $x\ge2$. The errors between true and predicted values are numerically smaller than 0.1 for all even states. Hence, the  measure appears to  grow super-exponentially fast. The function $\log(\pi(x))$ for the odd states grows  at a comparable rate as that for the even  states, but not with the same regularity. 

Additionally, we computed the difference in $\log(\pi(x))$ for  different values of $\phi^*$, showing again a distinction between odd and even states, see Figure \ref{fig:fig6} (top right). 

We used the linear approximation scheme to estimate the generating terms (which we know are not uniquely given in this case). Aslo here, an alternating pattern emerges with even indices producing the generator $(\pi(\sL_0),\pi(\sU_0))\approx (0.3764, 0.6235)$ ($n=100$), while odd indices producing the generator $\approx (0.4222, 0.5777)$ ($n=101$).  For each $n$, a unique solution is found. Computing   the corresponding $\phi^*=\phi(\sU_0+2)$ yields another approximation to $[\Psi_2(3),\Psi_1(3)]$, namely  $[1.5240, 2.7161]$, 
which is slightly larger than the previous approximation, $[1.5351, 2.6791]$. 
 By nature of the latter estimate, the first coordinate is increasing in  $n$, while the second is decreasing in $n$, hence the latter smaller  interval is closer to the true interval  $[\Psi_2(3),\Psi_1(3)]$, than the former. Figure \ref{fig:fig6} (bottom left)  also shows the estimated generating terms for different values of $n$, providing a band for $\pi(\sL_0)$ and $\pi(\sU_0)$ for which stationary measures exist, in agreement with Lemma \ref{le:stationary-measure}.

Finally, we computed the ratio of $\pi(n)$ to $\|\gamma(n)\|_1$ for different values of $\phi^*$, and observe that there appears to be one behavior for odd states and one for even. While one cannot infer the large state behavior of the  ratios in Figure \ref{fig:fig6}  (bottom right) from the figure, it is excluded by non-uniqueness of the measures,  that they both converges to zero, see Lemma  \ref{le:stationary-measure}.
\end{example}

\begin{figure}[h!]
\begin{center}
\includegraphics[width=4in]{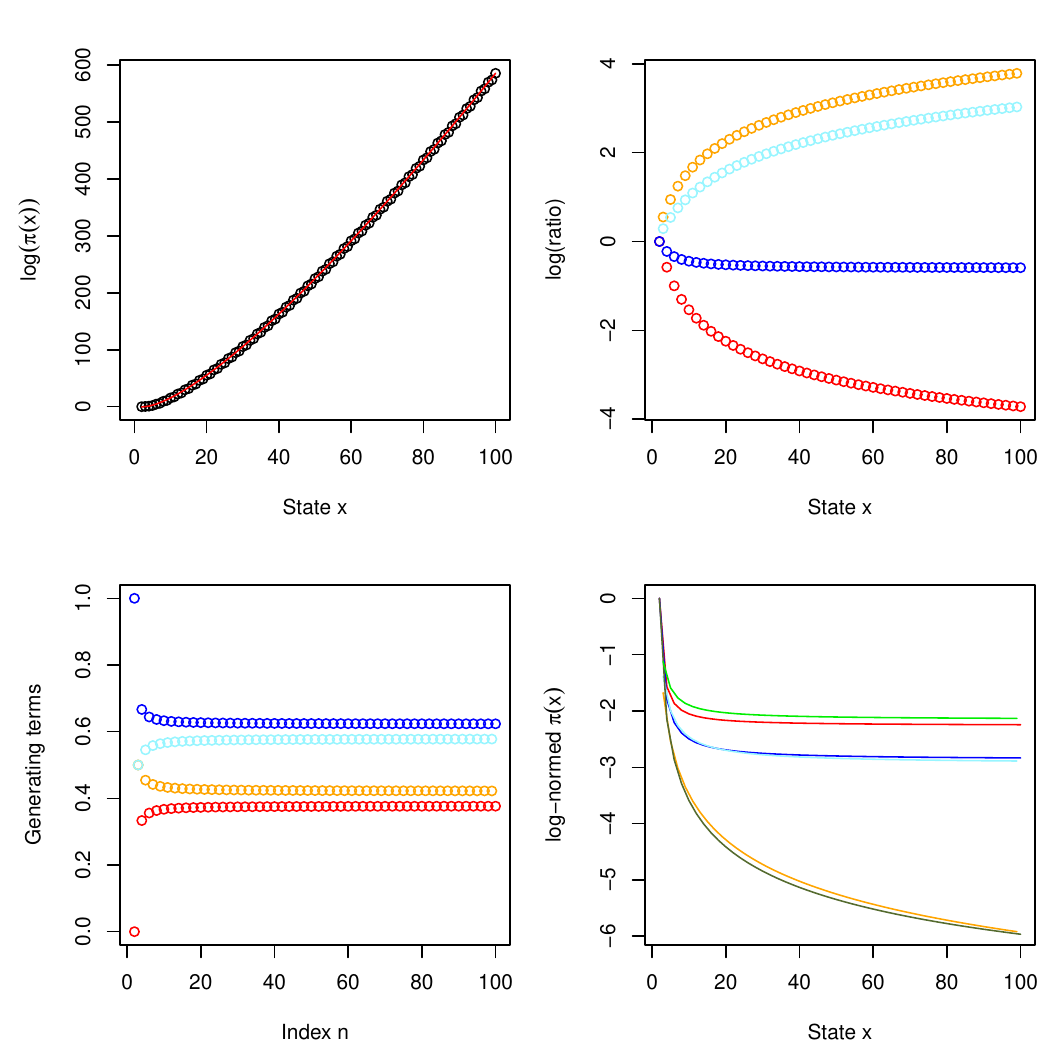}
\caption{Top left:  The  logarithm of the stationary measure with $\phi^*=2.67$. This value is in the upper end of the estimated interval $[\Psi_2(3),\Psi_1(3)]$,. The red line is the curve  fitted to the log-measure of the even states, using regression analysis. The stationary measures retrieved using the linear approximation scheme  in Section \ref{sec:convex} is visually identical to the plotted measure, not shown. Top right: The difference in $\log(\pi(x))$ between  the measure with $\phi^*=2.00$ (in the middle part of the interval) to that with $\phi^*=2.67$ (blue: even states, light blue: odd states),  and the log-measure between the measure with $\phi^*=1.54$ (in the lower end of the interval) to that with $\phi^*=2.67$  (red: even states, orange: odd states). An alternating pattern emerges. Bottom left: The generating terms estimated for different values of $n$, with red (even states)/orange (odd states) showing $\pi(\sL_0)$ and blue (even states)/light blue (odd states) showing $\pi(\sU_0)$. Bottom right: The normed measure, $\pi(n)/\|\gamma(n)\|_1$ for different values of $\phi^*$: $2.67$ (red: even states, orange: odd), $2.00$ (blue: even, light blue: odd) and $1.54$ (olive: even, green: odd). Dots are replaced by lines for visual reasons. }
\label{fig:fig6}
\end{center}
\end{figure}

\begin{example}\label{Ex-signed-stationary-measure} 
Consider the SRN with non-mass-action kinetics:
\[0\ce{<=>[\lambda_1][\lambda_{-1}]}\tS,\quad 2\tS\ce{->[\lambda_{-2}]}0\]
with transition rates given by:
\begin{align*}
\lambda_1(x) &= 2^x,\quad x\ge 0,\\
\lambda_{-1}(0)&=0,\quad \lambda_{-1}(1)=2,\quad \lambda_{-1}(x) = 2\cdot 4^{x-1} - 2^x,\quad x\ge 2,\\
\lambda_{-2}(0)& = \lambda_{-2}(1)=0,\quad \lambda_{-2}(x) = 2\cdot 4^{x-1},\quad x\ge 2.
\end{align*}
 For this SRN, $\omega_- = -2$, $\omega_+ = 1$, $s=0$,  $\sL_0=0$ and $\sU_0=1$. Hence, the conditions of  Proposition~\ref{pro-1} are satisfied. The underlying CTMC is from \cite{M63},  where it is shown that $\nu(x) = (-1/2)^{x+1}$ for $x\in\mathbb{N}_0$ is a signed invariant measure.  This measure  has generator $(\nu(0),\nu(1))=(-1/2,1/4)$. The space of signed invariant measures is $\sU_0-\sL_0+1=2$ dimensional, and a second linearly independent signed invariant measure has generator $(1,0)$. On the other hand, by the Foster-Lyapunov criterion \cite{meyn},  the process is positive recurrent, and hence admits a stationary distribution. Numerical computations show that the stationary distribution is concentrated on the first few states $x=0,\ldots, 3$. The uniqueness of the stationary distribution is also confirmed by  Corollary~\ref{cor:fraction} in that $H(\sU+2)=H(3)=\infty$ by a simple calculation.
\end{example}

\appendix

\section{Appendix}

\begin{lemma}\label{lem:ck}
Assume \rm{({\bf A1})}-\rm{({\bf A3})}. The following hold:
\begin{itemize}
\item[i)] $B_0=\{\om_-\}$, $\om_-\in B_k$ for $0\le k<-\om_-$, and $\om_+\in B_k$ for $\om_-\le k\le m_*$.
\item[ii)]  $c_0(\ell)<0$ for $\ell>\sU$ and $c_0(\ell)=0$ for $\ell\le \sU$.
\item[iii)] $c_k(\ell)>0$ for $-\om_-\le k\le m_*$ and $\ell\ge \si_{\om_+}$. Generally, $c_k(\ell)\ge0$, $\ell\in\N_0$.
\item[iv)]  $c_k(\ell)<0$ for $0\le k< -\om_-$ and $\ell\ge \si_{\om_-}=\sU+1$.  Generally, $c_k(\ell)\le0$, $\ell\in\N_0$.
\item[v)] If $c_k(\ell)>0$, then $c_k(\ell+1)>0$, and similarly, if $c_k(\ell)<0$, then $c_k(\ell+1)<0$ for $k\in\{0,\ldots,m_*\}$ and $\ell\in\N_0$.
\item[vi)] $f_k(\ell)>0$ for $-\om_-\le k\le m_*$ and $\ell-k\ge \si_{\om_+}$. 
\item[vii)]  $f_k(\ell)<0$ for $0\le k< -\om_-$ and $\ell-k\ge \si_{\om_-}=\sU+1$.
\item[viii)] If $f_k(\ell)>0$, then $f_k(\ell+1)>0$, and similarly, if $f_k(\ell)<0$, then $f_k(\ell+1)<0$ for $k\in\{0,\ldots,m_*\}$ and $\ell>\sU$.
\end{itemize}
\end{lemma}

\begin{proof}
i) Since $\om_-\le-1$, we have from \eqref{eq:Bk},
$$B_0=\{\om\in\Omega\,|\,(\om_-+1/2)(\om -\om_--1/2)>0\}=\{\om\in\Omega\,|\, \om \le\om_-\}=\{ \om_-\}.$$
For $-\om_-\le k\le m_*$, we have $k'=\om_-+k+1/2>0$ and $\om-k'\ge \om-(\om_-+m_*+1/2)=\om-\om_++1/2>0$ if $\om=\om_+$. Hence, $\om_+\in B_k$. For $0\le k<-\om_-$, we have $k'=\om_-+k+1/2<0$ and  $\om-k'\le\om-(\om_-+1/2)<0$ if $\om=\om_-$. Hence, $\om_-\in B_k$. 
ii) Since $\text{sgn}(\om_-+1/2)=-1$, then for $\ell>\sU=\si_{\om_-}-1 $, we have 
$c_0(\ell)=-\lambda_{\om_-}(\ell)<0$, see \eqref{eq:c}. Likewise, for $\ell\le \sU$. iii) and iv) The sign follows similarly to the proof of ii). Using i), the conditions on $\ell$ ensure $\lambda_{\om_+}(\ell)>0$ and $\lambda_{\om_-}(\ell)>0$, respectively, yielding the conclusions. v) It follows from \rm{({\bf A2})}. vi)-viii) Similar to iii)-iv) using \eqref{eq:f} and \eqref{eq:Bk}.
\end{proof}

For a matrix $D=(d_{ij})\in\R^{n\times n}$, the $j$-th column, $j=1,\ldots,n$, is {\em weakly  diagonally dominant} (WDD)   if $|d_{jj}|\ge\sum_{i\neq j}|d_{ij}|$ and \emph{strictly diagonally dominant} (SDD) if $|d_{jj}|>\sum_{i\neq j}|d_{ij}|$. In particular, a matrix is WDD if all columns are WDD. A WDD matrix is weakly chain diagonally dominant (WCDD) if for each WDD column, say $j$, which is not SDD, there exists an SDD column $k$ such that there is a directed path from  vertex $k$ to  vertex $j$ in the associated digraph. Every WCDD matrix is non-singular \cite{A18}.

\begin{lemma}\label{lem:echelon}
Assume \rm{({\bf A1})}-\rm{({\bf A3})}. Then, the row echelon form $G$ of $H$, as defined in \eqref{eq:H}, exists.
\end{lemma}

\begin{proof}
Recall that
\begin{align*}
H_{m,n}=\delta_{m,n}-\frac{\lambda_{(m-n)}(n)}{\sum_{\omega\in \Omega}\lambda_{\omega}(m)},\quad m=0,\ldots,\sL-1,\quad n=0,\ldots,U-1.
\end{align*}
Note the following: if $\om<-m$, then $m<\si_\om$ and hence  $\lambda_{\om}(m)=0$. Hence, we might replace $\sum_{\omega\in \Omega}\lambda_{\omega}(m)$ by $\sum_{k=-m}^\infty\lambda_{k}(m)$ in $H_{n,m}$.
 
Define $a_{m,n}=\lambda_{(m-n)}(n)$ for $n,m\in\Z$. Then the row echelon form $-H$ restricted to its first $\sL$ columns  is invertible, that is, if
\[\begin{pmatrix}\vspace{1em}
  1&-\frac{a_{0,1}}{\sum_{k=0}^\infty a_{k,0}}&\ldots&-\frac{a_{0,\sL-1}}
  {\sum_{k=0}^\infty a_{k,0}}\\ \vspace{1em}
  -\frac{a_{1,0}}{\sum_{k=0}^\infty a_{k,1}}&1&\ldots&-\frac{a_{1,\sL-1}}
  {\sum_{k=0}^\infty a_{k,1}}\\ \vspace{1em}
  \vdots&\vdots&\ddots&\vdots\\ \vspace{1em}
  -\frac{a_{\sL-1,0}}{\sum_{k=0}^\infty a_{k,\sL-1}}&-\frac{a_{\sL-1,1}}{\sum_{k=0}^\infty
  a_{k,\sL-1}}&\ldots&1
\end{pmatrix}\]
is invertible.  Multiplying the above matrix by  the invertible diagonal matrix
$${\rm diag}\lt(\sum_{k=0}^\infty a_{k,0},\sum_{k=0}^\infty a_{k,1},\ldots,
\sum_{k=0}^\infty a_{k,\sL-1}\rt)$$
 on the left side and its inverse on the right side gives a column diagonally dominant matrix
 \[A=\begin{pmatrix}\vspace{1em}
  1&-\frac{a_{0,1}}{\sum_{k=0}^\infty a_{k,1}}&\ldots&-\frac{a_{0,\sL-1}}
  {\sum_{k=0}^\infty a_{k,\sL-1}}\\ \vspace{1em}
  -\frac{a_{1,0}}{\sum_{k=0}^\infty a_{k,0}}&1&\ldots&-\frac{a_{1,\sL-1}}
  {\sum_{k=0}^\infty a_{k,\sL-1}}\\ \vspace{1em}
  \vdots&\vdots&\ddots&\vdots\\ \vspace{1em}
  -\frac{a_{\sL-1,0}}{\sum_{k=0}^\infty a_{k,0}}&-\frac{a_{\sL-1,1}}{\sum_{k=0}^\infty   a_{k,1}}&\ldots&1
\end{pmatrix}.\]
Note that $\sum_{k=1}^\infty a_{k,0}=\sum_{k=1}^\infty \lambda_{k}(0)>0$.   Furthermore, by ($\rm\mathbf{A2}$) the following property holds:
$a_{m,n}>0$ implies $a_{m+1,n+1}>0$ for  any $m,n\in\Z$, or by contraposition, $a_{m+1,n+1}=0$ implies $a_{m,n}=0$ for any $m,n\in\Z$.
Hence,
$$\sum_{k=1}^\infty a_{k+\sL-1,\sL-1}=\sum_{k=\sL}^\infty a_{k,\sL-1}>0,$$
 and therefore,
$$\sum_{k=0}^{\sL-1}a_{k,\sL-1}< \sum_{k=0}^\infty a_{k,\sL-1}.$$
Consequently, the $\sL$-th column sum of $A$ is positive, implying   $A$ is WDD matrix  with SDD column $\sL$.

By Lemma~\ref{le-1}, $A$ is invertible, and hence the row reduced echelon form exists.
\end{proof}

\begin{lemma}\label{le-1}
Let $n\in\N$  and assume $D$ is an $n\times n$ matrix, such that
\[ D=\begin{pmatrix}\vspace{1em}
  1&-\frac{d_{0,1}}{\sum_{k=0}^\infty d_{k,1}}&\ldots&-\frac{d_{0,n-1}}
  {\sum_{k=0}^\infty d_{k,n-1}}\\ \vspace{1em}
  -\frac{d_{1,0}}{\sum_{k=0}^\infty d_{k,0}}&1&\ldots&-\frac{d_{1,n-1}}
  {\sum_{k=0}^\infty d_{k,n-1}}\\ \vspace{1em}
  \vdots&\vdots&\ddots&\vdots\\ \vspace{1em}
  -\frac{d_{n-1,0}}{\sum_{k=0}^\infty d_{k,0}}&-\frac{d_{n-1,1}}{\sum_{k=0}^\infty d_{k,1}}
  &\ldots&1
\end{pmatrix},\]
where $d_{i,j}\ge0$ for $i\in\N_0$, $j=0,\ldots,n-1$,   $\sum_{i=1}^\infty d_{i,0}>0$, and $d_{i+1,j+1}=0$ implies $d_{i,j}=0$ for any $i\in\N_0$, $j=0,\ldots,n-2$. Then $D$ is WCDD, and thus non-singular.
\end{lemma}

\begin{proof}
Number rows and columns from $0$ to $n-1$. If $n=1$, then the statement is trivial. Hence assume $n>1$.

Fact 1: If column $j>0$ sums to zero, then column $j-1$ sums to zero. Indeed, that column $j$ sums to zero is equivalent to
$\sum_{i=n}^\infty d_{i,j}=0$, which by the property of $d_{i,j}$, implies that \begin{equation}\label{*}
\sum_{i=n}^\infty d_{i-1,j-1}=0.
\end{equation} Hence also  $\sum_{i=n}^\infty d_{i,j-1}=0$ and column $j-1$ sums to zero. Consequently, if column $j$ is WDD, but not SDD, then all columns $0,\ldots,j$ are WDD, but not SDD.

Fact 2: If column $j>0$ sums to zero, then from \eqref{*} it holds that $ d_{n-1,j-1}=0$. Inductively, using in addition Fact 1, $d_{i,k}=0$ for $i-k\ge n-j$, corresponding to a lower left triangular corner of size $j$.
\end{proof}

\section*{Acknowledgements}

The second author was supported by the Novo Nordisk Foundation (Denmark), grant NNF19OC0058354. 
The third author was supported by  TUM University Foundation, the Alexander von Humboldt Foundation, Simons Foundation, and a start-up funding from the University of Hawai'i at M\={a}noa.

The work was initiated when all authors were working at the University of Copenhagen. 
\red{We are grateful to the AE and   two anonymous reviewers whose comments and suggestions improved the presentation of the manuscript.}

The authors declare no competing interests.

\end{document}